\documentclass{article}
\usepackage{graphicx}
\usepackage{mathptmx}
\usepackage{amsmath}
\usepackage{amssymb}
\usepackage{amsthm}
\usepackage{pst-node}
\usepackage{tikz-cd} 
\usepackage{bm}
\usepackage[sorting=none]{biblatex}
\usepackage[a4paper, margin=2.5cm]{geometry}
\usepackage{authblk}

    \theoremstyle{plain}
    \newtheorem{theorem}{Theorem}
    \newtheorem{lemma}[theorem]{Lemma}
    
    \newtheorem{corollary}[theorem]{Corollary}
    
    \theoremstyle{remark}

    \newtheorem{definition}{Definition}

\newcommand{\Hml}{H}
\newcommand{\Hpq}{\Hml_{pq}}
\newcommand{\Hqp}{\Hml_{qp}}
\newcommand{\Hpp}{\Hml_{pp}}
\newcommand{\Hqq}{\Hml_{qq}}
\newcommand{\Hp}{\Hml_{p}}
\newcommand{\Hq}{\Hml_{q}}
\newcommand{\Oh}{\mathcal{O}}
\newcommand{\R}{\mathbb{R}}
\newcommand{\pp}{\tilde{p}_p}
\newcommand{\pq}{\tilde{p}_q}
\newcommand{\qp}{\tilde{q}_p}
\newcommand{\qq}{\tilde{q}_q}
\newcommand{\K}{\hat{H}}

\begin{document}


\title{Symplectic Error of Implicit Symplectic Integrators:\\ A Qualitative Structural Analysis}

\author{Mat\v{e}j Gajdo\v{s}}
\author{Ond\v{r}ej Brichta}
\author{V\'{a}clav Ku\v{c}era\thanks{Email: kucera@karlin.mff.cuni.cz}}
\renewcommand\Affilfont{\itshape\small}
\affil{
Charles University, Faculty of Mathematics and Physics, \protect\\ Sokolovská 83, 186 75 Praha 8, Czech Republic
}

\maketitle

\begin{abstract}
We study how inexact nonlinear solvers lead to a loss of exact symplecticity in the Symplectic Euler (SE) and Störmer–Verlet (SV) schemes when applied to general nonseparable Hamiltonian systems. These schemes are implicit and require nonlinear solvers in practice. Here, we consider a fixed number $M$ of fixed‑point iterations (FPI). While SE is exactly symplectic under exact solves, a finite $M$ gives only pseudo-symplecticity. Compared to previous results, we provide a more qualitative, block‑wise characterization of the induced pseudo‑symplecticity by analysing the resulting perturbations to the matrix of symplectic structure $J$. We prove that the perturbed matrix $\tilde{J}$ is skew‑symmetric, that one diagonal block vanishes identically (depending on the SE variant), and that the remaining blocks are $O(h^{M+1})$ perturbations of their counterparts in $J$, with time step $h$. A quadratic Hamiltonian example shows these bounds are sharp. Extending to compositions, we quantify how SV inherits distinct decay orders across different blocks of the symplectic defect. As a corollary, we show that the perturbation of volume preservation in phase space arises solely from the off‑diagonal blocks of $\tilde{J}$, and we bound the induced energy error along trajectories. Numerical experiments on a tokamak magnetic‑field Hamiltonian, where $q$-implicit SE is fully nonlinear (requiring FPI) but $p$-implicit SE is linearly implicit, confirm the sharpness of the theory and highlight the gap to the exactly symplectic counterpart.


\end{abstract}

\section{Introduction}
Structure‑preserving integrators, such as Symplectic Euler (SE) and St\"{o}rmer–Verlet (SV) schemes, are central to long‑time simulation of Hamiltonian systems because they retain key geometric features of the exact flow, most notably symplecticity and the resulting qualitative stability \cite{hairer-geom_int, sanz-serza-numerical-hamiltonian, hairer-ODE1}. When the Hamiltonian is separable, i.e. $\Hml = T(p) + V(q)$, these schemes are explicit and their exact symplecticity is transparent. In many applications of interest, however, the Hamiltonian is nonseparable and standard symplectic schemes become (nonlinearly) implicit. In practice, one then replaces the exact solution of the stage equations by a nonlinear solver such as fixed point or Newton iteration. While this is routine in implementations, it subtly breaks a key assumption behind the classical symplecticity proof: the stage equations are no longer solved exactly and symplecticity is no longer valid. Understanding precisely how this inexactness degrades structure preservation is the focus of the present work. We note that while this effect has been discussed in existing literature, it is not common knowledge among practitioners.

The effect of inexact solves on symplectic structure is commonly investigated through \emph{pseudo‑symplecticity}: the numerical flow $\Phi_h$ with time step $h$ satisfies
$$(D \Phi_h)^T J (D\Phi_h)= J + \Delta_h,$$ where $J$ is the skew-symmetric matrix of symplectic structure and $\Delta_h= \Oh(h^{r+1})$ measures the symplecticity defect, \cite{Aubry-pseudosymplecticity, Tan-norm_estimate}. Existing analyses typically provide quantitative bounds on $\Delta_h$ for Runge–Kutta schemes under various solver models (fixed‑point, Newton) and initialization strategies \cite{Aubry-pseudosymplecticity, Tan-norm_estimate, Calvo-pseudosymplectic-initializers, Aubry-pseudosymplecticity-FR}. These results illuminate how much symplecticity is lost, but not where the loss actually comes from structurally.

In this paper we revisit the Symplectic Euler and St\"{o}rmer–Verlet schemes and analyse the symplectic error more carefully in a qualitative manner. Specifically, we consider that the resulting nonlinear systems are solved by a fixed number $M$ of fixed point iterations. If we define $\tilde{J} = (D \Phi_h)^T J (D\Phi_h)$ for the resulting method, we show that $\tilde{J}$ is skew-symmetric, one of its diagonal blocks is exactly zero (which one depends on which version of SE we use), and that the remaining blocks are $O(h^{M+1})$ perturbations of the corresponding blocks in $J$. These results are demonstrated to be sharp by providing a specific example of a quadratic Hamiltonian for which the presented estimates are realized. On the example of the St\"{o}rmer-Verlet scheme, we then demonstrate how the result then carries on to methods based on composition of SE schemes, giving precise orders of decay of the individual blocks of the symplectic defect $\tilde{J}-J$.  

Understanding these more qualitative results on the perturbation of symplectic structure allows one to pinpoint the mechanism behind the loss of basic properties of the symplectic schemes, namely the loss of phase space volume preservation and energy preservation. We believe the refined analysis of the SE and SV schemes presented in this paper provides more qualitative insight into the behaviour of these methods rather than the purely quantitative pseudo-symplecticity estimates.

The structure of the paper is as follows. Section \ref{introduction_section} provides the necessary background and introduces the methods. The actual analysis of the structure of the perturbed symplectic matrix $\tilde{J}$ is performed in Section \ref{sec:Symplectic_Error}. An example of a quadratic Hamiltonian demonstrates the sharpness of these results. Consequences of the structural analysis are presented in Section \ref{sec:Consequences}. Namely, we show that the preservation of Lebesgue measure in the $(p,q)$-space is perturbed only due to the off-diagonal blocks of $\tilde{J}$ (i.e. only its skew-symmetric part), and we bound the error in energy preservation along solution trajectories. Finally, in Section \ref{sec:num_exp} we apply the results to the numerical solution of charged particle trajectories in a tokamak magnetic field, \cite{thompson-plasma, jackson-electrodynamics, cambon-tokamak_model}, which leads to a nonseparable Hamiltonian. For the sake of clarity, we move the more technical lemmas and proofs to the end Appendix of the paper.

\section{Symplectic structure and pseu\-do‑sym\-plec\-ti\-ci\-ty} \label{introduction_section}

The (autonomous) Hamiltonian system consists of $2N$ ordinary differential equations (ODEs)
\begin{equation}\label{eq:hamiltonian_eqs}
\begin{split}
z^{\prime}(t) &= J\nabla_z\Hml(z(t)),\quad t\in[0,T],\\
z(t_0)&=z^0\in\mathbb{R}^N\times\mathbb{R}^N,
\end{split}
\end{equation}
where $t_0\in\mathbb{R}$, $T>0$, $z = (q^T,p^T)^T\in\mathbb{R}^N\times\mathbb{R}^N$. The \emph{Hamiltonian} $\Hml=\Hml(z)$ is a mapping defined in a compact and convex set $\mathcal{D}\subset\mathbb{R}^{N}\times\mathbb{R}^N$, and $J\in\mathbb{R}^{2N\times2N}$ is the \emph{matrix of symplectic structure}
\begin{equation*}
    J:=\begin{pmatrix}
        O&I\\
        -I&O
    \end{pmatrix}
\end{equation*}
(with $I$ being the identity matrix). The vectors $q\in\mathbb{R}^N $ and $p\in\mathbb{R}^N$ are called the \emph{position} and \emph{momentum} vectors, respectively. 

For the sake of brevity, we will assume throughout this work that the Hamiltonian is sufficiently smooth, so that all derivatives are continuous on $\mathcal{D}$. 

It is well known that the above assumptions guarantee the well-posedness of the initial value problem \eqref{eq:hamiltonian_eqs}, i.e., for any $z^0 \in \hbox{Int}(\cal D)$, there exists a unique classical solution $z(t) = \varphi_t(z^0)$ for sufficiently small time $t$, which satisfies the initial condition $z(0) = \varphi_0(z^0) = z^0$. Moreover, it can be shown (see \cite{amann_ODE}, Chapter 2) that for a fixed sufficiently small $t$, the mapping $\varphi_t : {\cal D} \rightarrow \mathbb{R}^{2N}$, traditionally called the {\it flow} generated by the system \eqref{eq:hamiltonian_eqs}, is a symplectic mapping.
\begin{definition}
     Let $G\subset\mathbb{R}^N\times\mathbb{R}^N$ be open and $\phi:G\rightarrow\mathbb{R}^N\times\mathbb{R}^N$. We call $\phi$ a \emph{symplectic mapping} (or a \emph{symplectomorphism}), if $\phi$ is a diffeomorphism and for every $y\in G$, it holds that
     \begin{equation}
       (D \phi(y))^T J (D\phi(y)) = J,
       \label{def:symplectic}
     \end{equation}
     where $D\phi(y)$ denotes the Jacobian matrix of the mapping $\phi$ at the point $y$.     
 \end{definition}

This property immediately implies that the determinant of $D\varphi_t$ is equal to 1. Thus, the Lebesgue measure $\lambda^{2N}$ on $\mathbb{R}^{2N}$ is $\varphi_t$-invariant, i.e.:
\begin{equation*}
    \lambda^{2N}(\varphi_t(M))=\lambda^{2N}(M)
\end{equation*}
for any measurable set $M\subset\mathcal{D}$. This identity is often referred to as the \emph{volume preservation property} and plays a key role, for example, in the context of ergodic theory (for details, see \cite{Petersen_1983}).

A more detailed discussion of the mathematical theory of Hamiltonian dynamical systems can be found, e.g., in the monographs \cite{sanz-serza-numerical-hamiltonian} or \cite{Meyer-Hml_dyn_sys}.

Next, we turn our attention to the general form of a numerical integrator. For simplicity we assume that $\varphi_t(\mathcal{D})\subseteq\mathcal{D}$ for all $t\in[0,T]$. Let $\lbrace t_k\rbrace^K_{k=0}$, where $t_0 = 0$ and $t_K = T$, denote an equidistant partition of the interval $[0, T]$ with step size $h$. The one-step numerical integrator is then defined via its \emph{numerical flow} $\Phi_h : z^k \mapsto z^{k+1}$, where $z^k\in\mathcal{D}$ approximates $z(t_k)$. This mapping approximates the action of the flow $\varphi_h : z(t_k) \mapsto z(t_{k+1})$ described above. 

In the following text, we will use the notation $\Phi_h^{\Hml}$ to emphasize that the numerical integrator corresponds to Hamiltonian system \eqref{eq:hamiltonian_eqs} with the Hamiltonian $\Hml$, and assume that all the approximations $z^0,\ldots, z^K$ lie in $\mathcal{D}$.

Because the exact flow $\varphi_t$ of \eqref{eq:hamiltonian_eqs} is symplectic, it is natural to require symplecticity for the discrete flow as well.

\begin{definition}
We say that the numerical integrator is \textit{symplectic}, if for every smooth-enough Hamiltonian $\Hml$, the numerical flow $\Phi_h^{\Hml}$ given by the particular numerical scheme satisfies the condition (\ref{def:symplectic}).
\end{definition}

The Symplectic Euler method (SE), given by the numerical flow
\begin{equation*}
\Phi_h:\begin{pmatrix}
    q^k\\
    p^k
\end{pmatrix}\mapsto\begin{pmatrix}
    q^{k+1}\\
    p^{k+1} 
\end{pmatrix}, 
\end{equation*}
where
\begin{equation}\label{eq:symplectic_euler_method}
    \begin{split}
        p^{k+1} &:= p^k -h\nabla_q\Hml(q^k,p^{k+1}),\\
        q^{k+1} &:= q^k + h\nabla_p\Hml(q^k,p^{k+1}),
    \end{split}
\end{equation}
is the simplest example of a symplectic numerical integrator, which are generally implicit (see \cite{hairer-geom_int}, Chapter 6). This method also represents the simplest example of the so-called partitioned Runge–Kutta method (see \cite{hairer-geom_int}, Chapter 2). 
\begin{definition}
    Let $s\in\mathbb{N}$, $b_i,\widehat{b}_i\in\mathbb{R}$, $i=1,\ldots,s$, and $a_{ij},\widehat{a}_{ij}\in\mathbb{R}$, $i,j=1,\ldots,s$. A \emph{partitioned s-stage Runge-Kutta method} for the solution of (\ref{eq:hamiltonian_eqs}) is given by the numerical flow $\Phi_h:(q^T,p^T)^T\mapsto(\widetilde{q}^T,\widetilde{p}^T)^T$ constructed in the following way:
    \begin{equation}\label{Runge_Kutta_update}
       \begin{split}
           \widetilde{q}=q+h\sum^s_{i=1}b_ik_i,\\
           \widetilde{p}=p+h\sum^s_{i=1}\widehat{b}_il_i,
       \end{split}
    \end{equation}
    where the \emph{stages} $k_i$ and $l_i$ are given as a solution of the following system of (generally nonlinear) equations: 
    \begin{equation}\label{Runge_Kutta_stages}
        \begin{split}
            k_i &= \nabla_p\Hml\left(q+h\sum^s_{j=1}a_{ij}k_j,\,p+h\sum^s_{j=1}\widehat{a}_{ij}l_j\right),\quad i=1,\ldots,s,\\
            l_i &= -\nabla_q\Hml\left(q+h\sum^s_{j=1}a_{ij}k_j,\,p+h\sum^s_{j=1}\widehat{a}_{ij}l_j\right),\quad i=1,\ldots,s.
        \end{split}
    \end{equation}
\end{definition}

The Symplectic Euler method can be obtained from this general formulation by setting
\begin{equation*}
    s = 1,\quad a_{11}=0,\quad b_{1}=1,\quad \widehat{a}_{11}=1,\quad \widehat{b}_1=1.
\end{equation*}

The symplectic Euler method \eqref{eq:symplectic_euler_method} is an implicit scheme leading, in general, to a nonlinear system for $q^{k+1}, p^{k+1}$. We note that in the special case of \emph{separable} Hamiltonians (i.e. of the form $\Hml = T(p) + V(q)$), the SE scheme reduces to a fully explicit scheme. It is however the case of \emph{nonseparable} $\Hml$ which is of interest to us, since the SE scheme is in general implicit. It is therefore clear that some numerical method for solving nonlinear algebraic equations must be applied in general. In this work we focus on the traditionally recommended fixed point iteration (FPI) solver of nonlinear equations. In particular, we are interested exactly how the FPI solver applied to \eqref{eq:symplectic_euler_method} influences the symplectic condition \eqref{def:symplectic}. It will be shown that solving the system \eqref{eq:symplectic_euler_method} yields an explicit \textit{pseudo‑symplectic} scheme.

\subsection{Pseudo‑symplecticity}

The concept of pseudo‑symplecticity is introduced in \cite{Aubry-pseudosymplecticity} and \cite{Aubry-pseudosymplecticity-FR}. It aims to relax the symplectic condition to a broader range of numerical methods, while keeping some of the favorable theoretical properties of symplectic integrators. The definition follows.

\begin{definition}
We say that a numerical scheme is \textit{pseudo‑symplectic} of \textit{pseu\-do‑sym\-plec\-tic order} $r$, if for every smooth-enough Hamiltonian $\Hml$, the numerical flow $\Phi_h^{\Hml}$ given by the particular numerical scheme satisfies
\begin{equation}
(D \Phi_h^{\Hml})^T J (D\Phi_h^{\Hml}) = J + \Oh(h^{r+1}), \quad h \rightarrow 0.
\label{def:pseudo‑symplectic}
\end{equation}
\end{definition}

Therefore, bounding the error of the symplectic structure is equivalent to proving that the scheme being analysed is pseudo‑symplectic. By analysing this symplectic structure error for the FPI-Symplectic Euler method (FPI-SE), we will thus prove its pseudo‑symplecticity, and therefore also all the other useful properties stemming therefrom. 

Apart from the purely quantitative estimate \eqref{def:pseudo‑symplectic}, we explore in more detail the specific structure of the matrix $(D \Phi_h^{\Hml})^T J (D\Phi_h^{\Hml}) $ from \eqref{def:pseudo‑symplectic} in a more qualitative manner.

\section{Symplectic Error due to Implicitness}
\label{sec:Symplectic_Error}

Our aim in this section is to describe how the inexact, numerical solution of the implicit SE scheme  \eqref{eq:symplectic_euler_method} using FPI influences the symplecticity of this method. This analysis then yields additional results for other symplectic schemes, such as the Störmer-Verlet method.


Specifically, we show that applying the FPI method to the symplectic Euler method forces the whole scheme to be pseudo‑symplectic, where its pseudo‑symplectic order grows linearly with the number of fixed point iterations. This in itself corresponds to proving an estimate of the form \eqref{def:pseudo‑symplectic}. However, we will also examine the matrix structure of $(D\Phi_h^{\Hml})^T J (D \Phi_h^{\Hml})$ in more detail, with $\Phi_h^{\Hml}$ being now the flow of the Symplectic Euler scheme together with the FPI method. The resulting structure of this matrix is more complicated than that of $J$, yet it is not simply an arbitrary $\Oh(h^{r+1})$ perturbation of $J$. We also demonstrate the quantitative and qualitative sharpness of the derived estimates on a specially chosen quadratic Hamiltonian. 

\subsection{Related work}
We already introduced the concept of pseudo‑symplecticity, which is the essential property from which the important (near) structure-preserving properties of the FPI-SE method follow. The topic is introduced in \cite{Aubry-pseudosymplecticity, Aubry-pseudosymplecticity-FR}, but also discussed e.g. in \cite{Calvo-pseudosymplectic-initializers}.

However, we are aware of only a few works analysing the error in symplectic structure itself due to the finite iterative solving of implicit schemes. The pseudo\-symplec\-tic\-ity of a general symplectic Runge-Kutta (RK) scheme utilising FPI method is briefly mentioned in both \cite{Aubry-pseudosymplecticity, Aubry-pseudosymplecticity-FR} without going into much detail. The work \cite{Aubry-pseudosymplecticity-FR} views the RK+FPI method as a new explicit RK method with a specific Butcher's table (see also \cite{hairer-ODE1} Section II.11 for this approach), which can then be directly analysed using the pseudo‑symplectic RK order conditions.

A more in-depth analysis of the symplectic structure error can be found in \cite{Tan-norm_estimate}. It estimates the error in the spectral norm for symplectic RK methods utilising the FPI method as
\begin{equation}\label{Odhad_symp_ch_M2}
    \left\lVert (D\Phi_h)^T J (D\Phi_h) -J \right\rVert_2 = \Oh(M \delta^{M+2}), \quad h \rightarrow 0,
\end{equation}
where $M$ is the number of FPI iterations, $\delta = Ch$ and $C>0$ is a constant dependent on the particular RK scheme and Hamiltonian $\Hml$. However, the work \cite{Tan-norm_estimate} again studies only symplectic Runge-Kutta integrators, a class of methods to which the Symplectic Euler scheme does not belong.  

The fact that the symplectic error in (\ref{Odhad_symp_ch_M2}) is of order 
$\Oh(h^{M+2})$ (as opposed to the results of this work and those reported in 
\cite{Aubry-pseudosymplecticity,Aubry-pseudosymplecticity-FR}) stems from the 
analytical strategy adopted in \cite{Tan-norm_estimate}. There, the update of 
the state variables given by (\ref{Runge_Kutta_update}) is treated separately 
from the computation of the stages defined by the system 
(\ref{Runge_Kutta_stages}) and the associated FPI solver. In contrast, the 
analyses in \cite{Aubry-pseudosymplecticity,Aubry-pseudosymplecticity-FR} (as 
well as in this work) consider these two steps in a unified manner.

\subsection{General notation and schemes}
First, we introduce the basic notation and then recall the definitions of the Symplectic Euler method and the Störmer–Verlet scheme. 

Consider an arbitrary smooth function $F = F(q,p)$ defined in $\mathcal{D}$. Then, we denote by $F_q(q,p)$ and $F_p(q,p)$ the gradient of $F$ with respect to $q$ and $p$, respectively. Furthermore, we will frequently utilize the following $N\times N$ matrices:
\begin{equation*}
    \begin{split}
    \Hpp(q,p)&:=\bigg(\frac{\partial^2\Hml}{\partial p_s\partial p_r}(q,p)\bigg)_{r,s}
    \qquad \Hqq(q,p):=\bigg(\frac{\partial^2\Hml}{\partial q_s\partial q_r}(q,p)\bigg)_{r,s}\\
    \quad\Hpq(q,p)&:=\bigg(\frac{\partial^2\Hml}{\partial p_s\partial q_r}(q,p)\bigg)_{r,s}\qquad \Hqp(q,p):=\Hpq^T(q,p),
    \end{split}
\end{equation*}
where $\Hml$ is the Hamiltonian. Note that $\Hpp$ and $\Hqq$ are symmetric matrices in $\mathrm{Int}(\mathcal{D})$. We will often omit the arguments for simplicity.

Now, consider an arbitrary differentiable numerical flow $\Theta_h$:
\begin{equation*}
    \Theta_h:\begin{pmatrix}
        q\\
        p
    \end{pmatrix}\mapsto\begin{pmatrix}
        \tilde{q}\\
        \tilde{p}
    \end{pmatrix}.
\end{equation*}
If we define the following auxiliary operation 
\begin{equation}\label{strange_commutator_def}
    [R,S] := R^T S - S^T R, \qquad R, S \in \R^{n \times n},
\end{equation}
then a straightforward calculation leads to the following identity
\begin{equation}
\tilde{J_{\Theta}} := (D\Theta_h)^T J (D\Theta_h) = \begin{pmatrix} [\tilde{q}_q, \tilde{p}_q] & A_{\Theta} \\ -A_{\Theta}^T & [\tilde{q}_p, \tilde{p}_p] \end{pmatrix},
\label{eq:perturbed_matrix}
\end{equation}
where $$ A_{\Theta} : = \qq^T \pp - \pq^T \qp.$$
Notice that the matrix \eqref{eq:perturbed_matrix} is always skew-symmetric.

Let $M\in\mathbb{N}$, then the general step of the {\it Symplectic ($p$-implicit) Euler method} combined with $M$ steps of the FPI solver is given by
\begin{align}
\begin{split}
p_0 &:= p, \\
p_{n+1} &:= p - h \Hq(q,p_n), \quad n = 0, 1, \ldots, M-1,\\
\tilde{p} &:= p_M, \\
\tilde{q} &:= q + h \Hp(q,\tilde{p}),
\end{split}
\label{eq:scheme}
\end{align}
while the general step of the 
{\it adjoint} (or {\it $q$-implicit}) {\it Symplectic Euler scheme} combined with $M$ steps of the FPI solver takes the form
\begin{align}
\begin{split}
q_0 &:= q, \\
q_{n+1} &:= q + h \Hp(q_n,p), \quad n = 0, 1, \ldots, M-1,\\
\tilde{q} &:= q_M, \\
\tilde{p} &:= p - h \Hq(\tilde{q},p).
\end{split}
\label{eq:adjoint_scheme}
\end{align}
The mappings
\begin{equation*}
    \Phi_h^{[M]}:\begin{pmatrix}
        q\\
        p
    \end{pmatrix}\mapsto\begin{pmatrix}
        \tilde{q}\\
        \tilde{p}
    \end{pmatrix}
\end{equation*}
and
\begin{equation*}
    \Psi_h^{[M]}:\begin{pmatrix}
        q\\
        p
    \end{pmatrix}\mapsto\begin{pmatrix}
        \tilde{q}\\
        \tilde{p}
    \end{pmatrix}
\end{equation*}
then denote the corresponding numerical flows of methods (\ref{eq:scheme}) and (\ref{eq:adjoint_scheme}), respectively.

Based on this notation, we are able to present one step of the {\it Störmer–Verlet scheme} with fixed point iteration. This scheme also has two versions:
\begin{equation}
\begin{split}
p_0 &:= p,\\
p_{n+1}&:= p -\frac{h}{2}\Hq(q,p_n), \quad n = 0, 1, \ldots, M_1-1,\\
\bar{p}&:=p_{M_1},\\
q_0&:=q + \frac{h}{2}\Hp(q,\bar{p}),\\
q_{n+1}&:= q + \frac{h}{2}\left(\Hp(q,\bar{p})+\Hp(q_n,\bar{p})\right), \quad n = 0, 1, \ldots, M_2-1,\\
\tilde{q} &:= q_{M_2}\\
\tilde{p}&:=\bar{p}-\frac{h}{2}\Hq(\tilde{q}, \bar{p}),
\end{split}
\label{eq:SV1_scheme}
\end{equation}
and (with reversed role of $q,p$)
\begin{equation}
\begin{split}
q_0 &:= q,\\
q_{n+1}&:= q +\frac{h}{2}\Hp(q_n,p), \quad n = 0, 1, \ldots, M_1-1,\\
\bar{q}&:=q_{M_1},\\
p_0&:=p-\frac{h}{2}\Hq(\bar{q},p),\\
p_{n+1}&:= p - \frac{h}{2}\left(\Hq(\bar{q},p)+\Hq(\bar{q},p_n)\right), \quad n = 0, 1, \ldots, M_2-1,\\
\tilde{p} &:= p_{M_2}\\
\tilde{q}&:=\bar{q}+\frac{h}{2}\Hp(\bar{q},\tilde{p}).
\end{split}
\label{eq:SV2_scheme}
\end{equation}
It is easy to see that the numerical flow of the scheme (\ref{eq:SV1_scheme}) is equal to the composition
\begin{equation}\label{SV_flow1}
    \Upsilon_h^{[M_1,M_2]}:= \Psi_{\frac{h}{2}}^{[M_2]}\circ\Phi_{\frac{h}{2}}^{[M_1]},
\end{equation}
while the numerical flow of the scheme (\ref{eq:SV2_scheme}) is equal to the composition 
\begin{equation}\label{SV_flow2}
    \Lambda_h^{[M_1,M_2]}:=\Phi_{\frac{h}{2}}^{[M_2]}\circ\Psi_{\frac{h}{2}}^{[M_1]}, 
\end{equation}
just as the exact Störmer-Verlet flows are given by compositions of the two exact Symplectic Euler flows (see \cite{hairer-geom_int}, Section II.4, Example 4.5).
\subsection{Error estimates}

Assume that in \eqref{eq:perturbed_matrix}, the flow $\Theta_h$ is given by the numerical flow $\Phi_h^{[M]}$ of the $p$-implicit Symplectic Euler scheme given by \eqref{eq:scheme}, with $M$ being the number of fixed point iterations. Our first two results then estimate how the antidiagonal and diagonal blocks in \eqref{eq:perturbed_matrix} differ from the corresponding blocks in the symplectic matrix $J$. The proofs of the theorems are given in the appendix.

We note that the results give $\Oh$-asymptotic error bounds for $h$ approaching zero while keeping $M$ fixed. The constants of the dominant error terms (which are hidden in the $\Oh$-notation) could be in principle recovered from the proofs of the following two theorems, but the theorems only provide the asymptotic behaviour in their current formulation.

\begin{theorem}
After $M$ steps of the method \eqref{eq:scheme}, the antidiagonal block of the matrix $J$ is perturbed as
$$A_{\Phi} = I + \Oh(h^{M+1}), \quad h \rightarrow 0.$$
\label{thm:antidiagonal_estimate}
\end{theorem}

Remarkably, one of the diagonal blocks in \eqref{eq:perturbed_matrix} for the flow $\Phi_h^{[M]}$ stays zero without any additional error induced by the FPI solver. 

\begin{theorem}\label{thm:diag_estimate}
Consider $M$ steps of the FPI solver in \eqref{eq:scheme}, then for the diagonal blocks of the matrix $\tilde{J}$ the following holds:
\begin{equation}\label{zero_diag_block}
  [\tilde{q}_p, \tilde{p}_p] = O   
\end{equation}
and 
\begin{equation}\label{nonzero_diag_block}
  [\tilde{q}_q, \tilde{p}_q] = \Oh(h^{M+1}), \quad h \rightarrow 0.   
\end{equation}
\end{theorem}

From the partial results of Theorems \ref{thm:antidiagonal_estimate} and \ref{thm:diag_estimate}, we immediately get the following description of the symplectic structure error of the $p$-implicit Symplectic Euler scheme \eqref{eq:scheme}.

\begin{corollary}\label{cor:form_SE_matrix}
After $M$ FPI steps of the $p$-implicit Symplectic Euler method \eqref{eq:scheme}, the resulting numerical flow $\Phi_h^{[M]}$ satisfies
$$(D \Phi_h^{[M]})^T J (D\Phi_h^{[M]}) = \begin{pmatrix}D & I+E \\ -(I+E)^T & O\end{pmatrix} = J + \begin{pmatrix}D & E \\ -E^T & O\end{pmatrix},$$
where both matrices $D$ and $E$ are of order $\Oh(h^{M+1})$, $h \rightarrow 0$.
\end{corollary}

In order to get similar results for the $q$-implicit Symplectic Euler scheme \eqref{eq:adjoint_scheme}, we will use a particular symplectic mapping $\zeta$ which swaps the role of position and momenta coordinates. Using the $q$-implicit scheme will then correspond to using the $p$-implicit scheme in the new coordinates, giving us a link between the symplectic errors of both schemes. The next theorem shows that this approach works even when the FPI solver is incorporated (i.e. even when the numerical flows themselves are potentially not symplectic).

\begin{theorem}\label{thm:canonical_transformation}
Consider $M$ steps of the adjoint $q$-implicit scheme \eqref{eq:adjoint_scheme} with Hamiltonian $\Hml$ and denote by $\Psi_h^{[M]}$ its numerical flow. Let 
$$\zeta: (q,p) \mapsto (Q, P) := (-p, q)$$ 
and $\K := \Hml \circ \zeta^{-1}$. Let $\Phi_h^{[M]}$ be the numerical flow of the $P$-implicit scheme \eqref{eq:scheme} with Hamiltonian $\K$. Then the following diagram commutes
\[ \begin{tikzcd}
(q,p) \arrow{r}{\Psi^{[M]}_h} \arrow[swap]{d}{\zeta} & (\tilde{q},\tilde{p}) \arrow{d}{\zeta} \\%
(Q,P) \arrow{r}{\Phi^{[M]}_h}& (\tilde{Q},\tilde{P}).
\end{tikzcd}
\]
\end{theorem}

Using the previous theorem, we obtain an analogous result to Corollary \ref{cor:form_SE_matrix} for the $q$-implicit Symplectic Euler scheme \eqref{eq:adjoint_scheme}.

\begin{corollary}\label{cor:form_adjoint}
After $M$ FPI steps of the adjoint ($q$-implicit) Symplectic Euler method \eqref{eq:adjoint_scheme}, the resulting numerical flow $\Psi_h^{[M]}$ satisfies
$$\big(D\Psi_h^{[M]}\big)^T J \big(D\Psi_h^{[M]}\big) = \begin{pmatrix}O & I + \hat{E}\\ -(I+\hat{E})^T & \hat{D} \end{pmatrix} = J + \begin{pmatrix}O & \hat E \\ - \hat{E}^T & \hat{D}\end{pmatrix},$$
where both matrices $\hat{D}$ and $\hat{E}$ are of order $\Oh(h^{M+1})$, $h \rightarrow 0$.
\end{corollary}

Corollaries \ref{cor:form_SE_matrix} and \ref{cor:form_adjoint} describe the symplectic error of both variants of the Symplectic Euler method. Using those results and the chain rule, similar descriptions can be obtained for more complicated composed schemes of the form
$$\Theta_{\hat{h}_1, h_1, \dots, h_k}^{[\hat{M}_1, M_1, \dots, M_k]} := \Phi_{h_k}^{[M_k]} \circ \Psi_{\hat{h}_k}^{[\hat{M}_k]} \circ \cdots \circ \Phi_{h_1}^{[M_1]} \circ \Psi_{\hat{h}_1}^{[\hat{M}_1]}.$$
This is illustrated by the following result for the Störmer-Verlet schemes \eqref{eq:SV1_scheme} and \eqref{eq:SV2_scheme}.

\begin{corollary}\label{cor:SV_est}
 Consider $M_1$ steps of the FPI solver in (\ref{eq:scheme}) and $M_2$ steps of the FPI solver in (\ref{eq:adjoint_scheme}). For the St\"ormer-Verlet scheme with numerical flow $\Upsilon_h^{[M_1,M_2]}$ defined in (\ref{SV_flow1}), the following identity holds:
 \begin{equation}\label{symp_perturbation_SV1}
    \big(D\Upsilon_h^{[M_1,M_2]}\big)^TJ \big(D\Upsilon_h^{[M_1,M_2]}\big)=J + \begin{pmatrix}P_{11} & P_{12} \\ P_{21}& P_{22} \end{pmatrix},
\end{equation}
where matrices $P_{11}, P_{12}$ and $P_{21}$ are of order $\Oh(h^{\min\lbrace M_1,M_2\rbrace+1})$, $h\rightarrow 0$, and matrix $P_{22}$ is of order $\Oh(h^{M_2+1})$, $h\rightarrow 0$.

An analogous identity holds for the matrix $\big(D\Lambda_h^{[M_1,M_2]}\big)^TJ \big(D\Lambda_h^{[M_1,M_2]}\big)$ associated with the numerical flow $\Lambda_h^{[M_1,M_2]}$ from \eqref{SV_flow2}. In this case,
the matrices $P_{12}, P_{21}$ and $P_{22}$ are of order $\Oh(h^{\min\lbrace M_1,M_2\rbrace+1})$, $h\rightarrow 0$, and matrix $P_{11}$ is of order $\Oh(h^{M_2+1})$, $h\rightarrow 0$.  
\end{corollary}

\subsection{Optimality of the estimate}
 Let us show that we cannot generally improve the bound given by Theorems \ref{thm:antidiagonal_estimate} and \ref{thm:diag_estimate}. Consider the following quadratic Hamiltonian as a model problem
\begin{equation} \label{opt_model_hamiltnian}
    \Hml(q,p):=\frac{1}{2}\sum^N_{i=1}\left(p^2_i+q^2_i\right)+\sum^N_{i=1}\sum_{j=i+1}^N\left(p_iq_j-2q_ip_j\right).
\end{equation}

Straightforward calculations lead to 
\begin{equation*}
    \frac{\partial\Hml}{\partial q_k}=q_k + \sum^{k-1}_{i=1}p_i-2\sum^{N}_{i=k+1}p_i,\quad \frac{\partial\Hml}{\partial p_k}=p_k + \sum^N_{i=k+1}q_i-2\sum^{k-1}_{i=1}q_i,\quad k = 1,2\ldots N.
\end{equation*}
Thus
\begin{equation*}
    \Hqq=\Hpp=I
\end{equation*}
and
\begin{equation}\label{Hessian_matrix_optim}
    \Hpq = \begin{pmatrix}
0 & -2 & -2 & \cdots & -2 \\
1 & 0 & -2 & \cdots & -2 \\
1 & 1 & 0 & \cdots & -2 \\
\vdots & \vdots & \vdots & \ddots & \vdots \\
1 & 1 & 1 & \cdots & 0
\end{pmatrix}
\end{equation}
We will denote this matrix by $\Xi$. 

Before stating the theorem, let us introduce the following convenient notation for the the skew-symmetric part of an arbitrary square matrix
 \begin{equation}
    [A]_{\mathrm{skew}}:=\frac{1}{2}\left(A-A^T\right),\qquad A\in\mathbb{R}^{n \times n}. \label{eq:skew-part}
\end{equation} 

 \begin{theorem}\label{lemma:diagonal_block_optimality}
  Consider $M$ steps of the FPI solver in the Symplectic (p-implicit) Euler method \eqref{eq:scheme} applied to the Hamiltonian (\ref{opt_model_hamiltnian}). Then
 \begin{equation*}
      [\tilde{q}_q, \tilde{p}_q] =2(-1)^Mh^{M+1}\big[\Xi^M\big]_{\mathrm{skew}}
 \end{equation*}
 and
 \begin{equation*}
     A_{\Phi} = I +(-1)^{M}h^{M+1}\Xi^{M+1}.
 \end{equation*}
For method (\ref{eq:adjoint_scheme}), a similar result is also valid.
 \end{theorem}
The considered Hamiltonian is not merely an artificial construction. Mixed linear terms of the type appearing in (\ref{opt_model_hamiltnian}) can, for instance, arise from the action of a constant magnetic field, represented by a vector potential of the form $\bm{A}=(-\alpha y,\beta x,0)^T$, on a charged particle (see \cite{goldstein-mechanics}). The particular form of the Hamiltonian associated with this system can be found in Subsection \ref{subsection_char_par}.
  
\section{Consequences}
\label{sec:Consequences}
\subsection{Pseudo‑symplecticity and volume-preservation property}
In Section \ref{introduction_section}, we mentioned the invariance of the Lebesgue measure with respect to the flow generated by the system (\ref{eq:hamiltonian_eqs}) (the volume-preservation property), which is a consequence of the symplecticity of the flow. As we have already proven, the Symplectic Euler method combined with the FPI solver possesses only pseudo‑symplecticity. Therefore, a perturbation of this measure invariance is expected. 

The following theorem provides a quantitative description of the violation of the volume-preservation property. Moreover, the proof of this theorem shows that the behaviour of the evolution of the measure does not depend on the diagonal block $[\tilde{q}_q, \tilde{p}_q]$ in (\ref{eq:perturbed_matrix}) (i.e. it is independent of the error given in Theorem 2).
\begin{theorem}\label{volume_pr_p_psesy}
Let $G\subset \mathcal{D}$ be a measurable set. After $M$ steps of the FPI solver in (\ref{eq:scheme}) and (\ref{eq:adjoint_scheme}), the Lebesgue measure of the set $G$ is changed as 
\begin{equation}\label{volume_perturb_SE_1}
    \lambda^{2N}(\Phi_h^{[M]}(G)) = \lambda^{2N}(G) + \Oh(h^{M+1}),\quad h\rightarrow0
\end{equation}
and
 \begin{equation}\label{volume_perturb_SE_2}
    \lambda^{2N}(\Psi_h^{[M]}(G)) = \lambda^{2N}(G) + \Oh(h^{M+1}),\quad h\rightarrow0.
\end{equation}
Moreover, let $M_1,M_2\in\mathbb{N}$, then 
\begin{equation}\label{volume_perturb_SV1}
    \lambda^{2N}(\Upsilon_h^{[M_1,M_2]}(G)) = \lambda^{2N}(G) + \Oh(h^{\min\lbrace M_1,M_2\rbrace+1}),\quad h\rightarrow0.
\end{equation}
and
\begin{equation}\label{volume_perturb_SV2}
     \lambda^{2N}(\Lambda_h^{[M_1,M_2]}(G)) = \lambda^{2N}(G) + \Oh(h^{\min\lbrace M_1,M_2\rbrace+1}),\quad h\rightarrow0.
\end{equation}
\end{theorem}

Measure preservation is one of the key assumptions in ergodic theory; therefore, this result should be taken into account when applying ergodic theory to numerical simulations. One possible application is long‑time averaging in the context of molecular dynamics simulations (see e.g. \cite{ergodicity_num_sim} for details).

\subsection{Pseudo‑symplecticity and energy conservation}

For the autonomous problem \eqref{eq:hamiltonian_eqs}, the exact flow $\varphi_h$ preserves the value of $\Hml$ along the solution trajectory, i.e. $\Hml \circ \varphi_h = \Hml$, cf. \cite{Meyer-Hml_dyn_sys, hairer-geom_int}. Symplectic schemes mimic this behaviour by $\Hml \circ \Phi_h \approx \Hml$ with a bounded time-independent error over a long time interval -- this is often referred to as the \textit{energy conservation property}. These error bounds are typically derived using the backward analysis using the formal-series modified equation
$$z' = f_0(z) + hf_1(z)+h^2 f_2(z) + \dots,$$
which can be shown to also be Hamiltonian, i.e. $f_i = J \nabla\Hml_i$ for some functions $\Hml_0, \Hml_1, \dots$, cf. \cite{hairer-geom_int}.

As the use of the $M$-step FPI solver for the Symplectic Euler method corresponds to using a pseudo‑sym\-plec\-tic scheme of pseudo‑symplectic order $M$, the results about pseudo‑sym\-plec\-tic energy conservation can be directly applied to such a scheme. As mentioned in Section 2.3 of \cite{Aubry-pseudosymplecticity}, the standard proof of the modified equation being Hamiltonian can be applied in the order $r$ pseudo‑symplectic case, with the caveat that it holds only for the first $r$ terms of the modified equation, which then has the form
$$z' = J\nabla\big(\Hml_0(z) + h\Hml_1(z) + \dots + h^r \Hml_r(z)\big) + h^{r+1} f_{r+1}(z) + \dots$$
It follows that for the FPI-SE schemes \eqref{eq:scheme} and \eqref{eq:adjoint_scheme} (or pseudo‑symplectic schemes in general), the modified equation cannot be truncated arbitrarily, and thus in general it cannot be truncated optimally, so the strongest results about energy near-conservation in general might not hold, contrary to exact symplectic integrators.

To be more concrete, assume that an integrator $\Phi_h$ of order $p$ has a formal-series modified equation, which is after $K$-order truncation also Hamiltonian of the form
\begin{equation}
z' = J \nabla(\Hml(z) + h^p \Hml_p(z) + \dots + h^K \Hml_K(z)) = J \nabla \Hml^{[K]}(z).
\label{eq:truncated_modified_eq}
\end{equation}

The estimates given in \cite{hairer-geom_int}, Chapter IX, Theorem 8.1 utilise the triangle inequality
$$|\Hml(z_n) - \Hml(z_0)| \leq |\Hml(z_n) - \Hml^{[K]}(z_n)| + |\Hml^{[K]}(z_n) - \Hml^{[K]}(z_0)| + |\Hml^{[K]}(z_0) - \Hml(z_0)|$$
and note that $|\Hml - \Hml^{[K]}| = \Oh(h^p)$, $h \rightarrow 0$ under the assumption of boundedness of $\Hml_p + h\Hml_{p+1} + \dots + h^{K-p} \Hml_K$ (e.g. by restricting to a compact set and continuous $\Hml_i$). This gives
$$|\Hml(z_n) - \Hml(z_0)| = |\Hml^{[K]}(z_n) - \Hml^{[K]}(z_0)| + \Oh(h^p), \quad h \rightarrow 0,$$
so the actual time-dependent drift of the Hamiltonian is encapsulated in the error term $|\Hml^{[K]}(z_n) - \Hml^{[K]}(z_0)|$.

This is the term that is often estimated as ``negligible over exponentially long time intervals" \cite{hairer-geom_int}, at least under the assumption that the whole modified equation is Hamiltonian, so that an optimal truncation index $K$ can be chosen. 

To get a more general estimate without the need of optimal truncation, let $\varphi^{[K]}_t$ be the exact flow of \eqref{eq:truncated_modified_eq}. Under similar assumptions as in the mentioned Theorem 8.1 in \cite{hairer-geom_int}, one can combine the usual bound $\Vert\Phi_h - \varphi_h^{[K]}\Vert = \Oh(h^{K+1})$, followed by the Lipschitz assumption 
$$|\Hml^{[K]} \circ \Phi_h - \Hml^{[K]} \circ \varphi_h^{[K]}| = \Oh(1) \cdot \Vert\Phi_h - \varphi_h^{[K]}\Vert = \Oh(h^{K+1}), \quad h \rightarrow 0,$$
with the preservation of Hamiltonian under its flow $\Hml^{[K]} \circ \varphi^{[K]}_h = \Hml^{[K]}$ to get
\begin{align*}
|\Hml^{[K]}(z_n) - \Hml^{[K]}(z_0)| &\leq \sum_{i=0}^{n-1} |\Hml^{[K]}(z_{i+1}) - \Hml^{[K]}(z_i)| =\\
&=\sum_{i=0}^{n-1} |(\Hml^{[K]} \circ \Phi_h) (z_i) - (\Hml^{[K]} \circ \varphi_h^{[K]})(z_i)| =\\
&=\sum_{i=0}^{n-1} \Oh(h^{K+1}) = \Oh(nh^{K+1}).
\end{align*}
This, together with noting $t=nh$ as the integration time, gives the estimate
\begin{equation}
|\Hml(z_n) - \Hml(z_0)| = \Oh(h^p) + \Oh(th^K), \quad h \rightarrow 0.
\label{eq:energy-drift-estimate}
\end{equation}

As stated before, $M$ steps of the FPI solver used on the Symplectic Euler method results in a pseudo‑symplectic scheme of order $M$, which in particular means that the modified equation of the scheme is Hamiltonian for orders $K \leq M$. The estimate \eqref{eq:energy-drift-estimate} then gives the best asymptotic estimate
$$|\Hml_{\text{SE}}(z_n) - \Hml_{\text{SE}}(z_0)| = \Oh(h) + \Oh(th^M), \quad h \rightarrow 0.$$
Similar bounds can be obtained for schemes composed from several FPI-SE schemes.


\section{Numerical experiments}
\label{sec:num_exp}
In this section, we test all the derived results on the example of a specially chosen nonseparable Hamiltonian taken from plasma physics. This Hamiltonian is specific in that the $q$-implicit SE scheme is nonlinearly implicit, while the $p$-implicit scheme is only linearly implicit, thus bypassing the need for FPI, and can be therefore considered as an effectively explicit method.

\subsection{Charged particle in a tokamak magnetic field}
\label{subsection_char_par}

We consider the movement of a charged particle in a particular toroidal magnetic field. In particular, we will directly show how the actual perturbed matrix $\tilde{J}$ changes with respect to the values of $h$ and $M$. The energy (Hamiltonian) drift due to the finite number of FPI steps will also be empirically obtained.

Several reasons motivate us to study this particular physical problem: First, the problem is somewhat more complicated than the numerical experiments with harmonic oscillators, pendula, and Keplerian motion typically found in the standard literature when testing symplectic methods. Second, the electromagnetic Hamiltonian is prone to be nonseparable (i.e. not of the form $\Hml = T(p) + V(q)$), forcing the $q$-implicit scheme to be fully implicit. On the other hand, the $p$-implicit scheme can be viewed as basically explicit, which lets us compare the \emph{pseudo‑symplectic} behaviour of one Symplectic Euler scheme to the \emph{symplectic} behaviour of the other Symplectic Euler scheme.

To see the implicitness of the schemes, suppose we have a Hamiltonian of the form $\Hml = \frac{1}{2} \Vert \bm{p} - \bm{A} \Vert^2$ with dimensionless momentum $\bm{p}$ and vector potential $\bm{A}$, which is a function of position only. Note that we now boldface the vector quantities for clarity, as is customary in physics.

We have
$$\Hml_q= -(D_q\bm{A})^T \cdot( \bm{p}-\bm{A}), \qquad \Hml_p = \bm{p} - \bm{A}. $$
Thus, we usually cannot solve the implicit $q$-update
$$\bm{\tilde{q}} = \bm{q} + h \big(\bm{p} - \bm{A}(\bm{\tilde q})\big)$$
directly, so a root finding algorithm has to be incorporated. The FPI $q$-implicit scheme reads
\begin{align}
\begin{split}
\bm{q_0} &:= \bm{q}, \\
\bm{q_{n+1}} &:= \bm{q} + h \big(\bm{p} - \bm{A}(\bm{q_n})\big), \\
\bm{\tilde{q}} &:= \bm{q_M}, \\
\bm{\tilde p} &:= \bm{p} + h (D_q \bm{A}(\bm{\tilde q}))^T \cdot \big(\bm{p} - \bm{A}(\bm{\tilde q})\big).
\end{split}
\label{eq:elmag_implicit}
\end{align}

For the $p$-implicit scheme, the implicit update
$$\bm{\tilde p} = \bm{p} + h(D_q \bm{A}(\bm{q}))^T \cdot \big(\bm{\tilde p} - \bm{A}(\bm{q})\big),$$
can be rewritten as
\begin{equation}
\big(I - hD_q \bm{A}(\bm q)\big)^T \bm{\tilde p} = \bm{p} - h(D_q \bm{A}(\bm q))^T \cdot \bm{A}(\bm q),
\label{eq:elmag_explicit}
\end{equation}
which is a linear system that can be solved directly up to rounding errors, i.e. this scheme, together with the update
$$\bm{\tilde q} := \bm{q} + h\big( \bm{\tilde p} - \bm{A}(\bm q)\big),$$
can be viewed as explicit.


For numerical experiments, we use a simplified model of the magnetic field in a tokamak described in \cite{cambon-tokamak_model}. This uses the vector potential
$$\bm{A} = B_0 \frac{F(r)}{\rho} \bm{\hat{\varphi}} - B_0 R \log\left( \frac{\rho}{R} \right) \bm{\hat{z}},$$
where
$$F(r) := \int_0^r f(\lambda) \text{d}\lambda, \qquad f(r):= \frac{r}{Rq(r)}, \qquad q(r) := \frac{1+ar}{1+r^2},$$
cf. \cite{cambon-tokamak_model} -- equations 3, 2, and 14; here $(\bm{\hat{\rho}}, \bm{\hat{\varphi}}, \bm{\hat{z}})$ are the local orthonormal basis vectors of the cylindrical coordinate system (in \cite{cambon-tokamak_model}, $\xi$ is used instead of $\rho$); $r$ is the distance from the main toroidal axis (which is a circle with radius $R$), $a$ is a distance parameter and $B_0$ is a parameter corresponding to the strength of the magnetic field; see Fig. \ref{fig:tokamak_geometry}.

\begin{figure}
\centering
\includegraphics{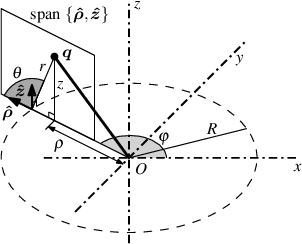}
\caption{Various coordinate systems used to describe the vector potential given in \cite{cambon-tokamak_model}.}
\label{fig:tokamak_geometry}
\end{figure}

The electromagnetic Hamiltonian has the form
$$\Hml = \frac{1}{2m} \Vert \bm{p} - Q \bm{A}\Vert^2 + Q\phi,$$
where $\phi$ is the electric potential \cite{goldstein-mechanics, jackson-electrodynamics}. In this particular case, we have $\phi=0$. From this Hamiltonian, the equations of motion can be directly obtained. For simplicity, we do so in cartesian coordinates, even though cylindrical coordinates might be more suitable in a practical scenario.

In order to be able to use somewhat realistic physical values (with different orders of magnitude), while still not amplifying the rounding errors too much, we employ a nondimensionalisation of the equations. We choose a characteristic length $L_0$ and time $T_0$, and then define
$$P_0 := m\frac{L_0}{T_0}, \qquad A_0 := \frac{P_0}{Q}, \qquad H_0 := \frac{P_0^2}{m}$$
and nondimensional quantities
$$\bm{q'} := \frac{\bm{q}}{L_0}, \quad \bm{p'} := \frac{\bm{p}}{P_0}, \quad \Hml' := \frac{\Hml}{H_0}, \quad \bm{A'}:=\frac{\bm A}{A_0}, \quad t' := \frac{t}{T_0}.$$
A simple calculation then shows that
$$\Hml' = \frac{1}{2} \Vert \bm{p'} - \bm{A'} \Vert^2$$
(note that this is the form which we discussed at the beginning of this section) and that the original Hamilton equations transform to the new Hamiltonian system
$$\frac{\text{d}z'}{\text{d}t'} = J\nabla_{z'} \Hml'(z'),$$
where $z' = (x', y', z', p_x', p_y', p_z')^T = ((\bm{q'})^T, (\bm{p'})^T)^T$. Because the system is again Hamiltonian, we can still employ symplectic integrators, which do not disturb the symplectic structure in the ideal scenario until taking into account the effects of floating-point arithmetic (FPA) and fixed point iteration.

In Fig. \ref{fig:tokamak_trajectory}, the trajectory of a charged particle in such a magnetic field is shown. The numerical integration was performed by the $q$-implicit Symplectic Euler scheme \eqref{eq:elmag_implicit} with 3 FPI iterations in each momentum update, and with parameter values $R = 5$~m, $a = 1$~m, $B_0 = 20$~mT, $m = 1.673 \cdot 10^{-27}$~kg, $Q = 1.602 \cdot 10^{-19}$~C, initial conditions
$$\begin{pmatrix}x_0\\y_0\\z_0\end{pmatrix} = \begin{pmatrix}5.1 \,\text{m}\\0\,\text{m}\\ 0.1\,\text{m}\end{pmatrix}, \qquad
\begin{pmatrix}p_{x0}\\p_{y0}\\p_{z0}\end{pmatrix} = \begin{pmatrix}10^{-23} \,\text{kg.m/s}\\10^{-23}\,\text{kg.m/s}\\ 10^{-21}\,\text{kg.m/s}\end{pmatrix},$$
characteristic dimensions 
$$L_0 = 2\pi R, \qquad T_0 = \frac{m}{QB_0},$$
and the step size $h = 0.1 \cdot T_0$. The choice of $T_0$ corresponds to the characteristic gyrofrequency $\sim 1/T_0$ of the particle \cite{thompson-plasma}.

\begin{figure}
\centering
\includegraphics[width=0.95\textwidth]{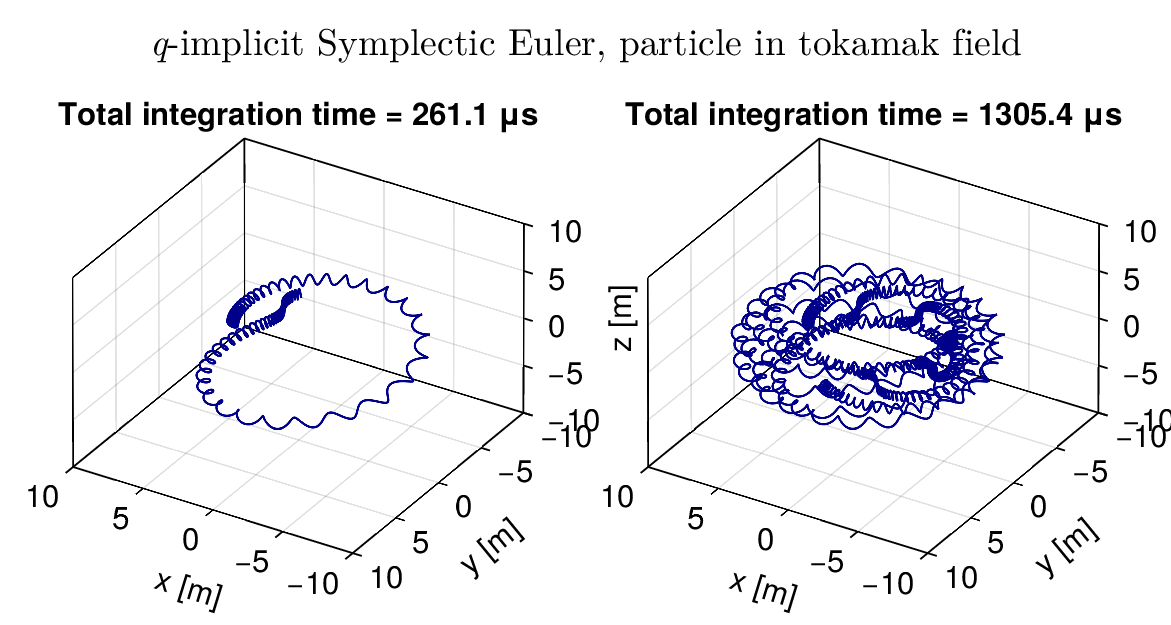}
\caption{Trajectory of a charged particle in a tokamak magnetic field for two different integration times. Integrated by $q$-implicit Symplectic Euler with $M=3$ FPI steps and step size $h=0.1 \cdot T_0$. Visually, this image is indistinguishable from the one obtained by integrating the same system using the $p$-implicit Symplectic Euler scheme \eqref{eq:elmag_explicit}.}
\label{fig:tokamak_trajectory}
\end{figure}

For the simulations, we used the Julia language \cite{Julia}. The code can be found at \url{https://github.com/gajdosmatej/symplectic_error}.

\subsubsection{Discrepancy in the symplectic structure matrix}

The result of a  single application of the $q$-implicit scheme \eqref{eq:elmag_implicit} with numerical flow $\Psi_h^{[M]}$ is analysed in this section. The flow $\Psi_h$ here updates the nondimensional quantities. The Jacobi matrices $D\Psi_h^{[M]}$ were calculated using forward automatic differentiation (symbolic calculation could also be used).

Denote $\delta_h^{[M]} := \big\Vert\, [\tilde q_p, \tilde p_p] \,\big\Vert_F$ and $\alpha_h^{[M]} := \Vert A_\Psi-I\Vert_F$, where $A_\Psi$ is the antidiagonal block from \eqref{eq:perturbed_matrix} (taken for the particular $h$ and $M$) and $\Vert \cdot \Vert_F$ is the Frobenius norm. Figures \ref{fig:discrepancy_diagonal} and \ref{fig:discrepancy_antidiagonal} both show how for a particular fixed $M$, those norms behave as $\approx Ch^{M+1}$ -- squares mark the actual computed values, and the dashed lines are found using the least squares (LS) method on the computed values in the log-log space for the dependency $\log \delta_h^{[M]} \approx \log(Ch^p) = p\log h + \log C$. These best-fit parameters are  presented in Table \ref{table:LS_params}.

\begin{figure}
\centering
\includegraphics[width=0.95\textwidth]{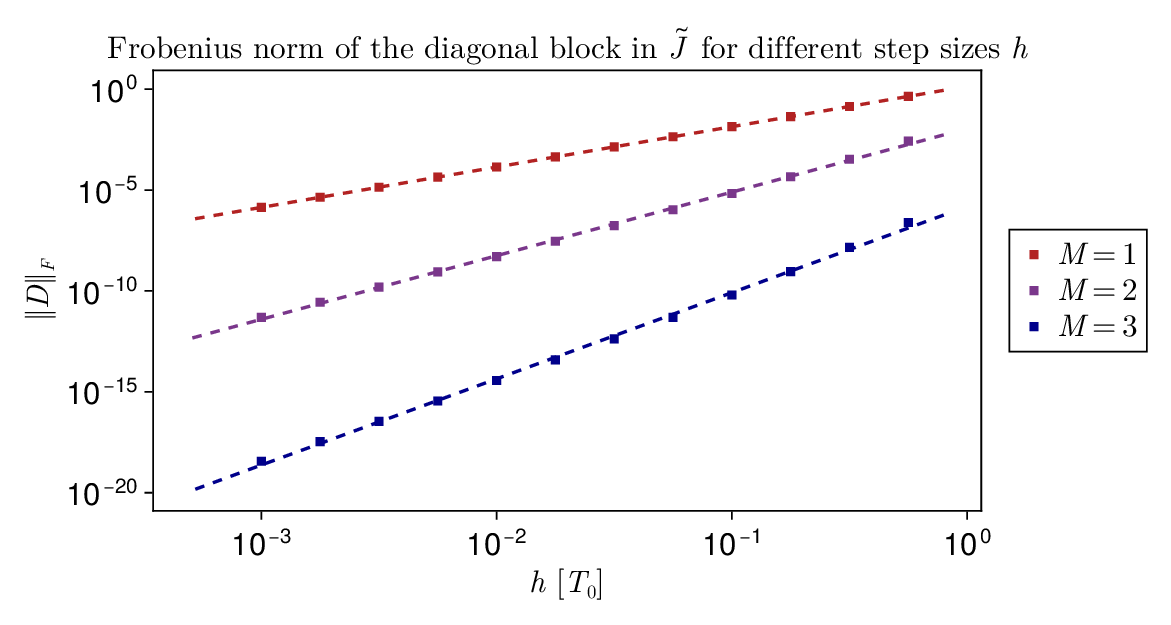}
\caption{The diagonal perturbation of the symplectic structure matrix $J$, measured in the Frobenius norm, evaluated for different step sizes $h$ and FPI steps $M$. One complete step of the scheme \eqref{eq:elmag_implicit} was used to integrate the charged particle problem after nondimensionalisation. On the vertical axis, $\Vert D \Vert_F$ corresponds to $\delta_h^{[M]}$.}
\label{fig:discrepancy_diagonal}
\end{figure}

\begin{figure}
\centering
\includegraphics[width=0.95\textwidth]{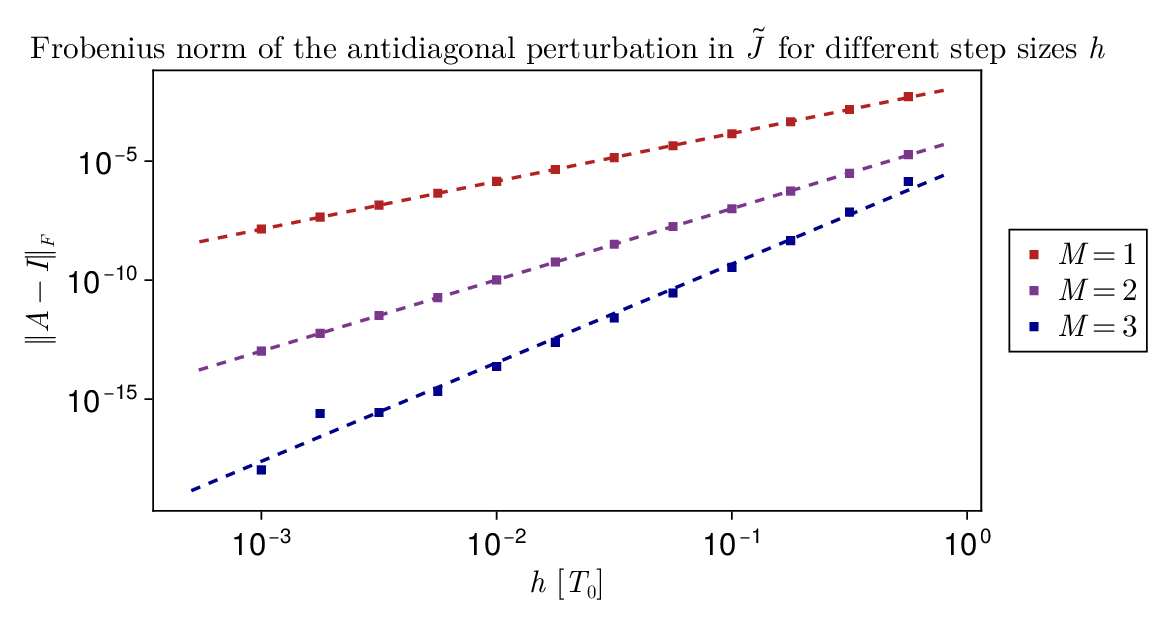}
\caption{The antidiagonal perturbation of the symplectic structure matrix $J$, measured in the Frobenius norm, evaluated for different step sizes $h$ and FPI steps $M$.}
\label{fig:discrepancy_antidiagonal}
\end{figure}

\begin{table}[h]
\centering
 \begin{tabular}{|c|c c|c c|} 
 \hline
 $M$ & $p_\delta$ & $C_\delta$ & $p_\alpha$ & $C_\alpha$ \\ [0.5ex] 
 \hline
 $1$ & $1.99997$ & $1.386$ & $2.01280$ & $1.502 \cdot 10^{-2}$\\
 $2$ & $3.15148$ & $1.120\cdot 10^{-2}$ & $2.99401$ & $9.926 \cdot 10^{-5}$ \\
 $3$ & $4.27162$ & $1.542\cdot 10^{-6}$ & $4.14529$ & $6.586 \cdot 10^{-6}$ \\
 \hline
 \end{tabular}
 \caption{Parameters found using LS linear fit in Figure \ref{fig:discrepancy_diagonal} (second and the third column, diagonal perturbation of the symplectic structure matrix) and in Figure \ref{fig:discrepancy_antidiagonal} (fourth and the fifth column, antidiagonal perturbation of the symplectic structure matrix).}
 \label{table:LS_params}
\end{table}

We also look at a specific instance of the perturbed $\tilde{J}$ matrix for a particular step size $h$ when the implicit step \eqref{eq:elmag_implicit} is employed. For $M=3$ and $h = 0.1 T_0$, we get
$$\begin{pmatrix}
    7.9\cdot 10^{-18} &  3.2\cdot 10^{-22} &  5.3\cdot 10^{-21} & 1.0 & -3.9\cdot 10^{-13} & 1.6\cdot 10^{-10} \\
    3.5\cdot 10^{-22} & 1.1\cdot 10^{-18} & -1.9\cdot 10^{-21} & 2.4\cdot 10^{-14} & 1.0 & 2.7\cdot 10^{-14} \\
    -5.4\cdot 10^{-21} & 2.2\cdot 10^{-21} & 1.9\cdot 10^{-22} & 2.5\cdot 10^{-12} & 4.5\cdot 10^{-13} & 1.0 \\
    -1.0 & -2.4\cdot 10^{-14} & -2.5\cdot 10^{-12} & 3.3\cdot 10^{-18} & -2.1\cdot 10^{-13} & 4.5\cdot 10^{-11} \\
    3.9\cdot 10^{-13} & -1.0 & -4.5\cdot 10^{-13} & 2.1\cdot 10^{-13} & -4.7\cdot 10^{-18} & 1.5\cdot 10^{-13} \\
    -1.6\cdot 10^{-10} & -2.7\cdot 10^{-14} & -1.0 & -4.5\cdot 10^{-11} & -1.5\cdot 10^{-13} & -9.7\cdot 10^{-18}
\end{pmatrix}.$$

Note that the theoretically zero block $[\tilde q_q, \tilde p_q]$ (the upper left 3x3 block) is not precisely zero due to the computer's finite precision, however the discrepancy is extremely small with entries on the order of $10^{-18} - 10^{-22}$. The theoretical skew-symmetry of $\tilde J$ is also preserved within rounding errors.

\subsubsection{Energy drift}

Finally, Fig. \ref{fig:energy_drift} shows the evolution of the Hamiltonian $\Hml$ during the particle motion both for the essentially explicit symplectic scheme \eqref{eq:elmag_explicit} and for the $q$-implicit scheme \eqref{eq:elmag_implicit}, for which the values $M=2$ and $M=3$ are used. The time is measured in multiples of $T_0$ as defined before. The simulation was performed on a long time interval with $3 \cdot 10^6$ time steps of the time integrators.

\begin{figure}[h]
\centering
\includegraphics[width=0.95\textwidth]{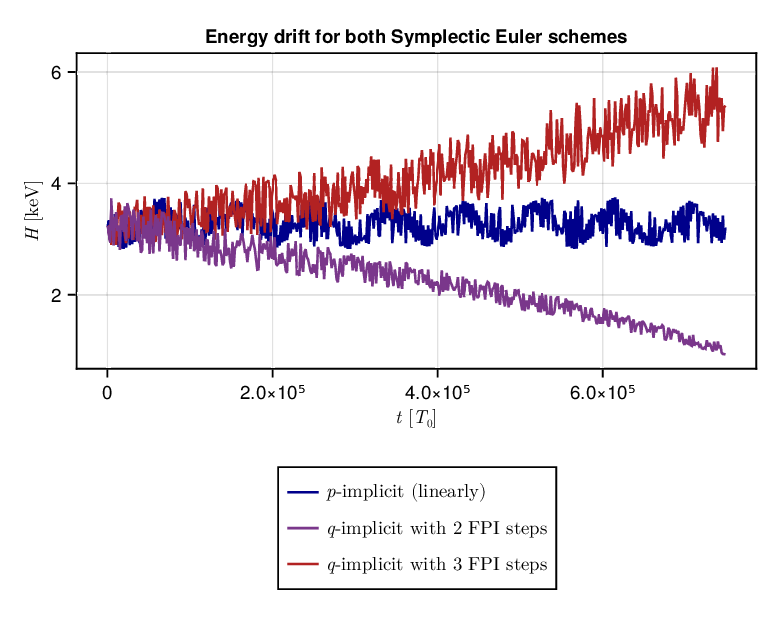}
\caption{Energy (Hamiltonian) drift of both the $q$-implicit (pseudo‑symplectic) and the $p$-implicit (symplectic) schemes over a long time interval ($h=0.25 T_0$, number of steps $3 \cdot 10^6$).}
\label{fig:energy_drift}
\end{figure}

The almost-explicit scheme \eqref{eq:elmag_explicit} exhibit a stable Hamiltonian development with a bounded time-independent error, just as is expected from a symplectic integrator. The implicit schemes \eqref{eq:elmag_implicit} get eventually disturbed during calculation, leading to an increasing error $|\Hml \circ \Psi_{nh}^{[M]} - \Hml_0|$.

\section{Conclusions}
In this work we have studied the loss of symplectic structure in the Symplectic Euler (SE) and St\"{o}rmer-Verlet schemes when applied to general nonseparable Hamiltonian dynamics. For such problems, these methods are implicit, leading to the need to solve the resulting nonlinear algebraic systems numerically. In the presented paper we consider the simplest fixed point iteration (FPI) method, which is the standard basic recommendation throughout the literature. This choice leads to the FPI-SE scheme. While the SE method is exactly symplectic if the implicit problems are solved exactly, the method is no longer symplectic if we apply only a finite number of fixed point iterations. This is a known problem, though it appears to be under-represented in the applied literature. 

Existing literature estimates the so-called pseudo-symplecticity of the resulting schemes, i.e. the magnitude of the perturbation of the skew-symmetric matrix of symplectic structure $J$ in the definition of symplecticity (i.e. the symplecticity defect). Such estimates are however mainly quantitative. In this work we present a more detailed qualitative analysis, showing how the individual blocks of $J$ are perturbed. Namely, if $\tilde{J} = (D \Phi_h)^T J (D\Phi_h)$, where $\Phi_h$ is the numerical flow of the method, we show that $\tilde{J}$ is skew-symmetric, one of its diagonal blocks is exactly zero (which one depends on which version of SE we use), and that the remaining blocks are $O(h^{M+1})$ perturbations of the corresponding blocks in $J$, where $h$ is the time step and $M$ is the number of fixed point iterations. We present an example of a quadratic Hamiltonian for which the presented estimates are optimal, hence cannot be improved in general. On the example of the St\"{o}rmer-Verlet scheme, we demonstrate how the result then carries on to methods based on composition of SE schemes, giving precise orders of decay of the individual blocks of the symplectic defect of $\tilde{J}-J$ (these turn out to be different for different blocks).  

As a consequence of the derived estimates, we show how the volume preservation property, which is characteristic of Hamiltonian flows, is perturbed in the SE and St\"{o}rmer-Verlet schemes. Namely, as a consequence of the structural analysis of $\tilde{J}$, we show that the preservation of Lebesgue measure in the $(p,q)$-space is perturbed only due to the off-diagonal blocks of $\tilde{J}$ (i.e. only its skew-symmetric part). 

Finally, we apply the results about the symplectic defects to bound the error in energy (i.e. value of the Hamiltonian) preservation along solution trajectories. 

The obtained results are tested on a model Hamiltonian which describes the movement of a charged particle in a  magnetic field of a tokamak. The Hamiltonian is nonseparable and special in that the $q$-implicit SE is fully nonlinear (thus requiring FPI), while the $p$-implicit SE scheme is only linearly implicit and can therefore be considered as an effectively explicit scheme in our setting. This allows for comparison between the loss of symplecticity in the FPI-SE scheme, and its  exactly symplectic counterpart. The obtained numerical results confirm all aspects of the presented theory.

Our results give a precise, block-wise structural description of the symplectic defect under inexact solves, pinpointing which components necessarily control measure and energy deviations and establishing optimal rates that are inherited by schemes based on composition. This provides a clean theoretical baseline for what can (and cannot) be improved in an inexact symplectic scheme, independent of algorithmic particulars. It brings to light subtle, often overlooked mechanisms in near‑symplectic behaviour due to the practical implementation of symplectic integrators, clarifying where exactly deviations originate and what they lead to.


\printbibliography
 
\section{Appendix}
\label{sec_Appendix}
\subsection{Appendix A. Proof of Theorem \ref{thm:antidiagonal_estimate} and Theorem~\ref{thm:diag_estimate}}

In this section, denote $\Hml^{[n]} := \Hml(q,p_n)$, where the values $p_n$ are given by the scheme \eqref{eq:scheme}.

\begin{lemma}
For the scheme \eqref{eq:scheme}, it holds that
\begin{equation*}
p_n - p_{n-1} = \Oh(h^n), \quad h \rightarrow 0, \quad \text{and} \quad \Hpq^{[n]} - \Hpq^{[n-1]} = \Oh(h^n), \quad h \rightarrow 0.\end{equation*}
\label{lemma:diff_momenta_hamilton}
\end{lemma}
\begin{proof}
For the difference of Hamiltonians, note that, using Taylor's polynomial,
$$\Hpq^{[n]} = \Hpq(q,p_n) = \Hpq(q,p_{n-1}) + \Hml_{ppq}(q, \xi)(p_n-p_{n-1}) = \Hpq^{[n-1]} + \Oh(p_n-p_{n-1})$$
for $\xi$ lying on the line segment given by $p_{n-1}$ and $p_n$.

The estimate of the difference of momenta is proved by induction. For $n=1$, direct substitution yields
$$p_1 - p_0 = (p-h\Hq(q,p))-p=-h\Hq(q,p) = \Oh(h).$$
Suppose the estimate holds for $n$. Then, it also holds that 
$$\Hpq^{[n]}-\Hpq^{[n-1]} = \Oh(p_n-p_{n-1}) = \Oh(h^n).$$
We thus have
\begin{align*}
p_{n+1} - p_n &= (p-h\Hq^{[n]}) - (p-h\Hq^{[n-1]}) = -h(\Hq^{[n]} - \Hq^{[n-1]}) = -h \Oh(h^n) =\\
&= \Oh(h^{n+1}), \quad h \rightarrow 0.
\end{align*}
\end{proof}

\begin{lemma}
For any $n \in \{1, \dots, M\}$,
$$\prod_{i=1}^n \Hpq^{[M-i]} - \prod_{i=0}^{n-1} \Hpq^{[M-i]} = \Oh(h^{M-n+1}), \quad h \rightarrow 0.$$
\label{lemma:diff_hamilton_prods}
\end{lemma}
\begin{proof}
We prove this lemma by induction over $n$ for fixed $M$. For $n=1$, Lemma \ref{lemma:diff_momenta_hamilton} yields
$$\Hpq^{[M-1]} - \Hpq^{[M]} = \Oh(h^{M}) = \Oh(h^{M-1+1}), \quad h \rightarrow 0.$$ 

Suppose that the estimate holds for $n <M$, then
\begin{align*}
\prod_{i=1}^{n+1} \Hpq^{[M-i]} - \prod_{i=0}^{n} \Hpq^{[M-i]} &= \prod_{i=1}^{n+1} \Hpq^{[M-i]} - \Big(\prod_{i=0}^{n-1} \Hpq^{[M-i]}\Big) \Hpq^{[M-n-1]} +\\
&\qquad+ \Big(\prod_{i=0}^{n-1} \Hpq^{[M-i]}\Big) \Hpq^{[M-n-1]} - \prod_{i=0}^{n} \Hpq^{[M-i]} =\\
&= \Big( \prod_{i=1}^{n} \Hpq^{[M-i]} - \prod_{i=0}^{n-1} \Hpq^{[M-i]} \Big) \Hpq^{[M-n-1]} +\\
&\qquad + \Big(\prod_{i=0}^{n-1} \Hpq^{[M-i]}\Big)\big(\Hpq^{[M-n-1]} - \Hpq^{[M-n]}\big) =\\
&= \Oh(h^{M-n+1}) \Hpq^{[M-n-1]} + \Big(\prod_{i=0}^{n-1} \Hpq^{[M-i]}\Big) \Oh(h^{M-n}) =\\
&= \Oh(h^{M-n}) = \Oh(h^{M-(n+1)+1}), \quad h \rightarrow 0,
\end{align*}
again using Lemma \ref{lemma:diff_momenta_hamilton}. 
\end{proof}

\begin{lemma}\label{lemma:diff_momenta_wrt_coordin}
    It holds that
    \begin{equation}\label{diff_momenta_wrt_coord}
        \pq =\sum_{n=1}^M (-h)^n \bigg(\prod_{i=1}^{n-1} \Hpq^{[M-i]}\bigg)\Hqq^{[M-n]}
    \end{equation}
    and
    \begin{equation}
        \pp = \sum_{n=0}^M (-h)^n \prod_{i=1}^n \Hpq^{[M-i]}. \label{eq:diff_momenta_wrt_momenta}
    \end{equation}
\end{lemma}
\begin{proof}
These assertions can be easily shown by mathematical induction based on (\ref{eq:scheme}) and the chain rule. 
\end{proof}

\begin{proof}[Proof of Theorem \ref{thm:antidiagonal_estimate}]
Differentiating \eqref{eq:scheme}, it holds that
$$\qq = I + h\Hqp^{[M]} + h\Hpp^{[M]} \pq, \qquad \qp = h\Hpp^{[M]} \pp.$$
Then, from the definition of $A_{\Phi}$,
$$A_{\Phi} = (I+h\Hqp^{[M]}+h\Hpp^{[M]} \pq)^T \pp - \pq^T \cdot h\Hpp^{[M]} \pp = (I+h \Hpq^{[M]}) \pp,$$
using $\Hqp^T = \Hpq$.

The $\pp$ expression from Lemma \ref{lemma:diff_momenta_wrt_coordin} then yields
\begin{equation} \label{perturbed_diagonal_block_expansion}
    \begin{split}
        A_{\Phi} &= (I+h\Hpq^{[M]}) \sum_{n=0}^M (-h)^n \prod_{i=1}^n \Hpq^{[M-i]} =\\
        &= I + \sum_{n=1}^M (-h)^n \prod_{i=1}^n \Hpq^{[M-i]} + h\Hpq^{[M]} \sum_{n=0}^{M-1} (-h)^n \prod_{i=1}^n \Hpq^{[M-i]} + h\Hpq^{[M]} (-h)^M \prod_{i=1}^M \Hpq^{[M-i]} \\
        &= I + \sum_{n=1}^M (-h)^n \prod_{i=1}^n \Hpq^{[M-i]} -\sum_{n=0}^{M-1} (-h)^{n+1} \prod_{i=0}^{n} \Hpq^{[M-i]} + \Oh(h^{M+1}) =\\
        &= I + \sum_{n=1}^M (-h)^n \left( \prod_{i=1}^n \Hpq^{[M-i]} - \prod_{i=0}^{n-1} \Hpq^{[M-i]} \right) + \Oh(h^{M+1}) = \\
        &= I + \sum_{n=1}^M (-h)^n \Oh(h^{M-n+1}) + \Oh(h^{M+1}) = I + \Oh(h^{M+1}), \quad h \rightarrow 0,
    \end{split}
\end{equation}
where we used Lemma \ref{lemma:diff_hamilton_prods}. 
\end{proof}

\begin{lemma}\label{lemma:Delta_est}
    Consider $M\geq2$ and denote
    \begin{equation*}
        \Delta^{[M]}_n:=\bigg(\prod_{j=0}^{n-2}\Hpq^{[M-j]}\bigg)\Hqq^{[M-n+1]}-\bigg(\prod_{j=1}^{n-1}\Hpq^{[M-j]}\bigg)\Hqq^{[M-n]}
    \end{equation*}
    for $n \in \{2,\ldots, M\}$. Then the following holds
    \begin{equation}\label{Delta_estimate}
        \forall n\in\lbrace 2,\ldots, M\rbrace:\quad\Delta^{[M]}_n = \Oh(h^{M-n+1}),\quad h\rightarrow0.
    \end{equation} 
\end{lemma}
\begin{proof}
Direct algebraic manipulation leads to
\begin{equation}\label{reformulation_aux_diff}
\begin{split}
    \Delta^{[M]}_n &= \bigg(\prod_{j=0}^{n-2}\Hpq^{[M-j]}\bigg)\left(\Hqq^{[M-n+1]}-\Hqq^{[M-n]}\right)+\\
    &\qquad+\bigg(\prod_{j=0}^{n-2}\Hpq^{[M-j]}-\prod_{j=1}^{n-1}\Hpq^{[M-j]}\bigg)\Hqq^{[M-n]}.
    \end{split}
\end{equation}
Similarly as in the proof of Lemma \ref{lemma:diff_momenta_hamilton}, we can observe that
\begin{equation*}
    \Hqq^{[M-n+1]}-\Hqq^{[M-n]}=\Oh(h^{M-n+1}),\quad h \rightarrow 0.
\end{equation*}
From Lemma \ref{lemma:diff_hamilton_prods} it directly follows that 
\begin{equation*}
   \prod_{j=0}^{n-2}\Hpq^{[M-j]}-\prod_{j=1}^{n-1}\Hpq^{[M-j]}=\Oh(h^{M-n+2}),\quad h\rightarrow 0.
\end{equation*}

The desired equality (\ref{Delta_estimate}) is then an easy consequence of plugging the above equalities into \eqref{reformulation_aux_diff} and the arithmetic of the big $\Oh$ notation. 

\end{proof}

\begin{proof}[Proof of Theorem \ref{thm:diag_estimate}]
Let us begin with the proof of identity (\ref{zero_diag_block}). This identity simply follows from the definition of the structure $[\cdot,\cdot]$ (see (\ref{strange_commutator_def})) and the symmetry of the matrix $\Hpp^{[M]}$
\begin{equation*}
    [\tilde{q}_p, \tilde{p}_p] = h\left(\Hpp^{[M]}\pp\right)^T\pp-\pp^T\left(h\Hpp^{[M]}\pp\right)=h\left(\pp^T\Hpp^{[M]}\pp-\pp^T\Hpp^{[M]}\pp\right)=O,
\end{equation*}
where we used the identity $\qp=h\Hpp^{[M]}\pp$.

Now, we turn our attention to the proof of the estimate (\ref{nonzero_diag_block}). Just as in the proof of Theorem \ref{thm:antidiagonal_estimate}, differentiating \eqref{eq:scheme} gives
\begin{equation*}
    \qq=I+h\Hqp^{[M]}+h\Hpp^{[M]}\pq.
\end{equation*}
Using this relation, the symmetry of the matrix $\Hpp^{[M]}$ and identity $(\Hpq^{[M]})^T=\Hqp^{[M]}$, one can obtain (recall the $[\cdot]_{\text{skew}}$ notation from \eqref{eq:skew-part})
\begin{equation*}
   [\tilde{q}_q, \tilde{p}_q] =2 \bigg [\left(I+h\Hpq^{[M]}\right)\pq\bigg]_{\mathrm{skew}}.
\end{equation*}
It is convenient to introduce the following notation
\begin{equation} \label{def_gamma}
    \Gamma:=\left(I+h\Hpq^{[M]}\right)\pq.
\end{equation}
Plugging (\ref{diff_momenta_wrt_coord}) into (\ref{def_gamma}) and re-arranging terms, we have 
\begin{equation}\label{equation_for_Gamma}
    \begin{split}
        \Gamma=&-h\Hqq^{[M-1]} +(-1)^Mh^{M+1}\bigg(\prod_{j=0}^{M-1}\Hpq^{[M-j]}\bigg)\Hqq^{[0]}+
        \sum^{M}_{n=2}(-1)^{n+1}h^n\Delta_n^{[M]}.
    \end{split}
\end{equation}
Therefore,
\begin{equation*}
    \begin{split}
        [\tilde{q}_q, \tilde{p}_q]=&2[\Gamma]_{\mathrm{skew}} =2(-1)^Mh^{M+1}\Bigg[\bigg(\prod_{j=0}^{M-1}\Hpq^{[M-j]}\bigg)\Hqq^{[0]}\Bigg]_{\mathrm{skew}}+\\
        &\qquad+2\sum^{M}_{n=2}(-1)^{n+1}h^n\big[\Delta_n^{[M]}\big]_{\mathrm{skew}},
    \end{split}
\end{equation*}
since $\Hqq^{[M-1]}$ is a symmetric matrix.

Using Lemma \ref{lemma:Delta_est}, the second term in this identity can be rewritten as follows:
\begin{equation*}
    \sum^{M}_{n=2}(-1)^{n+1}h^n\big[\Delta_n^{[M]}\big]_{\mathrm{skew}}=\sum^{M}_{n=2}(-1)^{n+1}h^n\Oh(h^{M-n+1})=\Oh(h^{M+1}),\quad h\rightarrow0.
\end{equation*}
The estimate \eqref{nonzero_diag_block} immediately follows. 
\end{proof}

\subsection{Appendix B. Proof of Theorem \ref{thm:canonical_transformation} and Corollary \ref{cor:form_adjoint}.}
\begin{proof}[Proof of Theorem \ref{thm:canonical_transformation}]
Denote the components of the inverse function $\zeta^{-1}$ as 
$$\zeta^{-1}(Q,P) = (q(P),p(Q)) = (P, -Q),$$
i.e. by functions $q$, $p$ mapping new coordinates $(Q,P)$ to corresponding old coordinates $(q,p)$. By the chain rule, it holds that
\begin{align*}
\K_Q(Q,P) &= (\Hml \circ \zeta^{-1})_Q(Q,P) = \Hml_p\big(q(P),p(Q)\big) \cdot \partial_Q p(Q) = - (\Hml_p \circ \zeta ^{-1})(Q,P), \\
\K_P(Q,P) &= (\Hml \circ \zeta^{-1})_P (Q,P) = \Hml_q\big(q(P),p(Q)\big) \cdot \partial_P q(P) = (\Hml_q \circ \zeta^{-1})(Q,P).
\end{align*}

Now, suppose that we have a state $(q,p)$ with corresponding new coordinates $(Q,P) = (-p, q)$ and we apply the $q$-implicit scheme \eqref{eq:adjoint_scheme} with Hamiltonian $\Hml$ to the old-coordinate state $(q,p)$, and the $P$-implicit scheme \eqref{eq:scheme} with Hamiltonian $\K$ to the new-coordinate state $(Q,P)$. This way, we obtain values $q_0,\dots,q_M, \tilde{p}$ and $P_0,\dots,P_M, \tilde{Q}$, respectively. 

We now show that $P_i = P(q_i)$ and $\tilde Q = Q(\tilde p)$, i.e. it does not matter if we first apply the scheme \eqref{eq:adjoint_scheme} and then transform to the new coordinates, or if we first transform and then apply the scheme \eqref{eq:scheme} in the new coordinates:

We defined $P_0 = P = q = q_0 = P(q_0)$. If for $i \in \{0, \dots, M-1\}$ it holds that $P_i = P(q_i)$, then
\begin{align*}
P_{i+1} &= P - h \K_Q(Q,P_i) = q - h(-\Hp \circ \zeta^{-1})(Q,P_i) =\\
&= q + h (\Hp \circ \zeta^{-1})(Q(p), P(q_i)) = q + h (\Hml_p \circ \zeta^{-1})(\zeta(q_i,p)) =\\
&= q + h\Hp(q_i,p) = q_{i+1} = P(q_{i+1}).
\end{align*}
Also,
\begin{align*}
\tilde Q &= Q + h\K_P (Q, \tilde P) = -p + h(\Hq \circ \zeta^{-1})(Q(p), P(\tilde q)) =\\
&=-p + h \Hq(\tilde q, p) = -\tilde p = Q(\tilde p).
\end{align*}
Thus the application of schemes $\eqref{eq:scheme}$ and $\eqref{eq:adjoint_scheme}$ commute with $\zeta$. 
\end{proof}

\begin{proof}[Proof of Corollary \ref{cor:form_adjoint}]
Just as in Theorem \ref{thm:canonical_transformation}, denote $(q,p)$ the old coordinates and $(Q,P)$ the new coordinates related by the mapping $\zeta$. Also again denote $\K = \Hml \circ \zeta^{-1}$. Let $\Psi_h^{[M]}: (q,p) \mapsto (\tilde q, \tilde p)$ be the adjoint, $q$-implicit scheme \eqref{eq:adjoint_scheme}, and let $\Phi_h^{[M]}:(Q,P) \mapsto (\tilde Q, \tilde P)$ be the $P$-implicit scheme \eqref{eq:scheme} in the new coordinates.

By \eqref{eq:perturbed_matrix}, we have
$$\big(D \Phi_h^{[M]}\big)^T J \big(D\Phi_h^{[M]}\big) = \begin{pmatrix} [\tilde{Q}_Q, \tilde{P}_Q] & A_{\K} \\ -A_{\K}^T & [\tilde{Q}_P, \tilde{P}_P] \end{pmatrix}$$
with $A_{\K} := \tilde Q_Q^T \tilde P_P - \tilde P_Q^T \tilde Q_P$. For the flow $\Psi_h^{[M]}$, we also have
$$(D\Psi_h^{[M]})^T J (D\Psi_h^{[M]}) = \begin{pmatrix} [\tilde{q}_q, \tilde{p}_q] & A \\ -A^T & [\tilde{q}_p, \tilde{p}_p] \end{pmatrix}$$
for $A = \tilde q_q^T \tilde p_p - \tilde p_q^T \tilde q_p$, again due to \eqref{eq:perturbed_matrix}.

Thanks to Theorem \ref{thm:canonical_transformation}, the derivatives of the new-coordinate flow are related to the derivatives of the old-coordinate flow via
$$D \Phi_h^{[M]} = D\zeta \cdot D\Psi_h^{[M]} \cdot D\zeta^{-1} = J \cdot D\Psi_h^{[M]} \cdot J^{-1},$$
which yields
\begin{align*}
\tilde Q_Q &= \tilde p_p, &\tilde P_Q &= - \tilde q_p,\\
\tilde Q_P &= - \tilde p_q, &\tilde P_P &= \tilde q_q.
\end{align*}

We can thus relate the blocks in $(D\Psi_h^{[M]})^T J (D\Psi_h^{[M]})$ to blocks in $\big(D \Phi_h^{[M]}\big)^T J \big(D\Phi_h^{[M]}\big)$. It holds that
$$[\tilde Q_Q, \tilde P_Q] = [\tilde q_p, \tilde p_p], \quad [\tilde Q_P, \tilde P_P] = [\tilde q_q, \tilde p_q], \quad A_{\K} = A^T.$$

Moreover, as $\Phi_h^{[M]}$ is formally given by \eqref{eq:scheme}, we can apply the previous Theorems \ref{thm:antidiagonal_estimate} and $\ref{thm:diag_estimate}$ or the Corollary \ref{cor:form_SE_matrix} to obtain
$$[\tilde q_p, \tilde p_p] = \Oh(h^{M+1}), \quad [\tilde q_q, \tilde p_q] = O, \quad A = I + \Oh(h^{M+1}), \quad h \rightarrow 0.$$

\end{proof}

\subsection{Appendix C. Proof of Corollary \ref{cor:SV_est}.}
\begin{proof}[Proof of Corollary \ref{cor:SV_est}]
From Corollary \ref{cor:form_SE_matrix} and Corollary \ref{cor:form_adjoint}, it follows that
\begin{equation}\label{pseudosympl_Phi_half}
    (D \Phi_{\frac{h}{2}}^{[M_1]})^T J (D\Phi_{\frac{h}{2}}^{[M_1]})  = J + \begin{pmatrix}D & E \\ -E^T & O\end{pmatrix},
\end{equation}
and
\begin{equation}\label{pseudosympl_Psi_half}
    \big[\big(D\Psi_{\frac{h}{2}}^{[M_2]}\big)\circ\Phi_{\frac{h}{2}}^{[M_1]}\big]^T J \big[\big(D\Psi_{\frac{h}{2}}^{[M_2]}\big)\circ\Phi_{\frac{h}{2}}^{[M_1]}\big]= J + \begin{pmatrix}O & \hat E \\ - \hat{E}^T & \hat{D}\end{pmatrix},
\end{equation}
where both matrices $D$ and $E$ are of order $\Oh(h^{M_1+1})$, $h \rightarrow 0$, and both matrices $\hat{D}$ and $\hat{E}$ are of order $\Oh(h^{M_2+1})$, $h \rightarrow 0$. 

Combining identities (\ref{pseudosympl_Phi_half}) and (\ref{pseudosympl_Psi_half}) yields
\begin{equation*}
\begin{split}
    \big(D\Upsilon_h^{[M_1,M_2]}\big)^TJ \big(D\Upsilon_h^{[M_1,M_2]}\big)&=J + \begin{pmatrix}D & E \\ -E^T & O\end{pmatrix}\\
     &\quad+(D \Phi_{\frac{h}{2}}^{[M_1]})^T \begin{pmatrix}O & \hat E \\ - \hat{E}^T & \hat{D}\end{pmatrix} (D\Phi_{\frac{h}{2}}^{[M_1]}).
\end{split}
\end{equation*}
The form of the perturbation in (\ref{symp_perturbation_SV1}) then follows straightforwardly from this identity.

The case of the numerical flow $\Lambda_h^{[M_1,M_2]}$ can be treated in an analogous manner. 
\end{proof}

\subsection{Appendix D. Proof of Theorem \ref{lemma:diagonal_block_optimality}}
\begin{lemma}\label{lemma:H^M_is_not_symmetric}
    Let $M\in\mathbb{N}$ and $\Xi$ be given by (\ref{Hessian_matrix_optim}). Then there exists a mapping $\xi^{(M)}:\mathbb{Z}\rightarrow\mathbb{Z}$ such that 
    \begin{equation*}
    (\Xi^M)_{ij} = \xi^{(M)}(i-j),\quad i,j = 1,\ldots,N.
    \end{equation*}
    i.e., $\Xi^M$ is a Toeplitz matrix. Moreover, this mapping satisfies the following relation
    \begin{equation}\label{auxiliary_relation_toeplic}
    -2\xi^{(M)}(l)=\xi^{(M)}(l-N),\quad l=1,\ldots,N-1
\end{equation}
and the matrix $\Xi^M$ is not symmetric.
\end{lemma}
\begin{proof}
    We proceed by induction over $M$. For $M=1$ we have
    \begin{equation*}
    \xi^{(1)}(l)=\begin{cases}
        -2 & \text{if }\  l<0,\\
         0 & \text{if }\  l=0,\\
         1 & \text{if }\  l>0
    \end{cases}
\end{equation*}
and (\ref{auxiliary_relation_toeplic}) obviously holds. 

Suppose that $\Xi^M$ is a Toeplitz matrix for an arbitrary $M\in\mathbb{N}$ and that the relation (\ref{auxiliary_relation_toeplic}) holds. Then, from this hypothesis and the definition of matrix multiplication, we have
\begin{equation*}
\begin{split}
    (\Xi^{M+1})_{i+1,j+1} &= \sum_{l=1}^N\xi^{(M)}(i+1-l)\xi^{(1)}(l-j-1)\\
    &=\sum_{l=-j}^{N-j-1}\xi^{(M)}(i-j-l)\xi^{(1)}(l)\\
    &=(\Xi^{M+1})_{ij}-\xi^{(M)}(i-N)-2\xi^{(M)}(i)\\
    &=(\Xi^{M+1})_{ij}
\end{split}
\end{equation*}
for all $i,j\in\lbrace1,\ldots,N-1\rbrace$. Thus, there exists a mapping $\xi^{(M+1)}:\mathbb{Z}\rightarrow\mathbb{Z}$ such that 
    \begin{equation*}
    (\Xi^{M+1})_{ij} = \xi^{(M+1)}(i-j),\quad i,j = 1,\ldots,N.
\end{equation*} 
Moreover, from the structure of $\Xi$, it follows that
\begin{equation*}
    \xi^{(M+1)}(l)=\left(\Xi\Xi^M\right)_{l+1,1}=\sum_{k=1}^{l}\xi^{(M)}(k-1)-2\sum_{k=l+2}^{N}\xi^{(M)}(k-1),\quad l=0,\ldots N-1
\end{equation*}
and 
\begin{equation*}
    \xi^{(M+1)}(l-N)=\left(\Xi\Xi^M\right)_{l,N}=\sum_{k=1}^{l-1}\xi^{(M)}(k-N)-2\sum_{k=l+1}^{N}\xi^{(M)}(k-N),\quad l=1,\ldots N,
\end{equation*}
where we used the convention that the sum over the empty set is equal to zero. By employing the induction hypothesis (\ref{auxiliary_relation_toeplic}), the following relation
 \begin{equation}
    -2\xi^{(M+1)}(l)=\xi^{(M+1)}(l-N),\quad l=1,\ldots,N-1
\end{equation}
can be easily derived. 
\end{proof}
\begin{proof}[Proof of Theorem \ref{lemma:diagonal_block_optimality}]
We show the result for the $p$-implicit Symplectic Euler method (\ref{eq:scheme}). By virtue of Theorem \ref{thm:canonical_transformation}, a similar assertion holds also for the adjoint $q$-implicit scheme (\ref{eq:adjoint_scheme}).

From (\ref{equation_for_Gamma}), we immediately obtain
\begin{equation*}
     \begin{split}
        \Gamma=&-hI +(-1)^Mh^{M+1} \Xi^M,
    \end{split}
\end{equation*}
since 
\begin{equation*}
    \begin{split}
        \Delta^{[M]}_n&=\bigg(\prod_{j=0}^{n-2}\Hpq^{[M-j]}\bigg)\Hqq^{[M-n+1]}-\bigg(\prod_{j=1}^{n-1}\Hpq^{[M-j]}\bigg)\Hqq^{[M-n]}=\\
        &=\Xi^{n-1}-\Xi^{n-1}=O,\quad n=2,\ldots,M.
    \end{split}
\end{equation*}
From Lemma \ref{lemma:H^M_is_not_symmetric}, it follows that 
\begin{equation*}
     [\tilde{q}_q, \tilde{p}_q] =2(-1)^Mh^{M+1}\big[\Xi^M\big]_{\mathrm{skew}}.
\end{equation*}

The antidiagonal blocks of the perturbed symplectic structure can be treated analogously. From (\ref{perturbed_diagonal_block_expansion}), it follows that 
\begin{equation*}
    \begin{split}
        A_{\Phi}&= I + \sum_{n=1}^M (-h)^n \left( \prod_{i=1}^n \Hpq^{[M-i]} - \prod_{i=0}^{n-1} \Hpq^{[M-i]} \right) + (-1)^Mh^{M+1}\Hpq^{[M]} \prod_{i=1}^M \Hpq^{[M-i]}\\
        &=I +(-1)^{M}h^{M+1}\Xi^{M+1}.
    \end{split}
\end{equation*}
\end{proof}
\subsection{Appendix E. Proof of Theorem \ref{volume_pr_p_psesy}}
\begin{lemma}\label{aux_alg_lemma_block_det}
    Let $A,B,C,$ and $D$ be  $d\times d$ complex matrices, then \begin{equation*} \det \begin{pmatrix}
            A & B\\
            C & D
        \end{pmatrix} = \det (AD-BC)
    \end{equation*}
    whenever at least one of the blocks $A,B,C$ and $D$ is equal to $O$.
\end{lemma}
\begin{proof}
For the proof, see \cite{block_det}. 
\end{proof}

\begin{proof}[Proof of Theorem \ref{volume_pr_p_psesy}]
The standard properties of the determinant and Lemma \ref{aux_alg_lemma_block_det} yield
\begin{equation*}
   \det\tilde{J}_{\Phi}= \det\left((D\Phi^{[M]}_h)^TJD\Phi^{[M]}_h\right) = \left(\det D\Phi^{[M]}_h\right)^2\det J=\left(\det D\Phi^{[M]}_h\right)^2,
\end{equation*}
where we used $\det J = 1$.

From Corollary \ref{cor:form_SE_matrix}, it follows that
\begin{equation*}
     \det \tilde{J}_{\Phi}=\det \begin{pmatrix} [\tilde{q}_q, \tilde{p}_q] & A_{\Phi} \\ -A_{\Phi}^T & O \end{pmatrix}=\left(\det A_{\Phi}\right)^2,
\end{equation*}
where we employed Lemma \ref{aux_alg_lemma_block_det}. Now, by comparing the left hand sides of the two expressions above, we obtain the following identity:
\begin{equation*}
    \left|\det D\Phi^{[M]}_h\right|=|\det A_{\Phi}|.
\end{equation*}
Let $G\subset \mathcal{D}$ be a measurable set. Then, according to the substitution theorem, we have
\begin{equation}\label{volume}
    \begin{split}
        \lambda^{2N}(\Phi_h^{[M]}(G)) = \int_{\Phi_h^{[M]}(G)}\mathrm{d}\tilde{z} = \int_{G}\left|\det D\Phi^{[M]}_h\right|\mathrm{d}z = \int_G|\det A_{\Phi}|\mathrm{d}z.
    \end{split}
\end{equation}

The determinant of a perturbed identity matrix can be expanded using the following well-known formula
\begin{equation*}
    \det (I+E) = 1 + \mathrm{Tr}\;E + \omega(E), 
\end{equation*}
where higher-order terms are hidden in $\omega(E)$.  From the proof of Theorem \ref{thm:antidiagonal_estimate}, it follows that $A_{\Phi} = I + E$ with
\begin{equation*}
    E=\sum_{n=1}^M (-h)^n \left( \prod_{i=1}^n \Hpq^{[M-i]} - \prod_{i=0}^{n-1} \Hpq^{[M-i]} \right) + h\Hpq^{M} (-h)^M \prod_{i=1}^M \Hpq^{[M-i]}.
\end{equation*}
Using standard properties of the matrix trace and Taylor expansion, it can easily be shown that $\mathrm{Tr}\;E=\Oh(h^{M+1})$, $h\rightarrow0$. Altogether, we obtain
\begin{equation}\label{determinant_DPhi}
    \left|\det D\Phi^{[M]}_h\right| = 1 + \Oh(h^{M+1}),\quad h\rightarrow 0.
\end{equation}
Thus, identity (\ref{volume_perturb_SE_1}) follows from (\ref{volume}). 

In order to obtain identity (\ref{volume_perturb_SE_2}), we proceed similarly. For the numerical flow $\Psi^{[M]}_h$, we also have 
\begin{equation}\label{determinant_DPsi}
\left|\det D\Psi^{[M]}_h\right|=1 + \Oh(h^{M+1}),\quad h\rightarrow 0.
\end{equation}
Integrating then yields (\ref{volume_perturb_SE_2}). 

Furthermore, Theorem \ref{thm:canonical_transformation} implies that
\begin{equation*}
\Psi^{[M]}_h=\zeta^{-1}\circ\Phi^{[M]}_h\circ\zeta,    
\end{equation*}
where we recall that $\Phi^{[M]}_h$ is associated with the Hamiltonian $\K$. Using the chain rule, identities (\ref{SV_flow1}) and (\ref{SV_flow2}), equations (\ref{determinant_DPhi}) and (\ref{determinant_DPsi}), and the fact that $\left|\det D\zeta\right|=1$ in $\mathcal{D}$, it follows that both determinants $\left|\det D\Upsilon_h^{[M_1,M_2]}\right|$ and $\left|\det D\Lambda_h^{[M_1,M_2]}\right|$ are equal to $1+\Oh(h^{\min{\lbrace M_1,M_2\rbrace}+1})$ as $h\to 0$. Finally, we apply the substitution theorem to complete the proof of identities (\ref{volume_perturb_SV1}) and~(\ref{volume_perturb_SV2}). 
\end{proof}
\end{document}